\documentclass[10pt]{article}
\usepackage[a4paper,centering,vscale=0.75,hscale=0.7]{geometry}

\usepackage{mathtools}
\usepackage{amsthm,amssymb}
\usepackage{mathrsfs}
\usepackage{color}
\usepackage{bbm}

\newtheorem{thm}{Theorem}[section]
\newtheorem{lemma}[thm]{Lemma}
\newtheorem{prop}[thm]{Proposition}

\newtheorem{defi}[thm]{Definition}
\theoremstyle{remark}
\newtheorem{rmk}[thm]{Remark}
\newtheorem*{rmk*}{Remark}
\newtheorem*{rmks*}{Remarks}

\renewcommand{\leq}{\leqslant}
\renewcommand{\geq}{\geqslant}
\newcommand{\dom}{\mathsf{D}}

\newcommand{\E}{\mathop{{}\mathbb{E}}}

\newcommand{\cF}{\mathscr{F}}
\newcommand{\cL}{\mathscr{L}}

\renewcommand{\P}{\mathbb{P}}
\newcommand{\erre}{\mathbb{R}}
\newcommand{\enne}{\mathbb{N}}

\def\cC{\mathscr{C}}
\def\cS{\mathscr{S}}

\newcommand{\longto}{\longrightarrow}

\DeclarePairedDelimiter\abs{\lvert}{\rvert}
\DeclarePairedDelimiter\norm{\lVert}{\rVert}
\DeclarePairedDelimiterX\ip[2]{\langle}{\rangle}{#1,#2}
\DeclarePairedDelimiterX\ipp[2]{\langle\!\langle}{\rangle\!\rangle}{#1,#2}

\DeclarePairedDelimiterX\oo[2]{]\!]}{[\![}{#1,#2}
\DeclarePairedDelimiterX\cc[2]{[\![}{]\!]}{#1,#2}
\DeclarePairedDelimiterX\co[2]{[\![}{[\![}{#1,#2}
\DeclarePairedDelimiterX\oc[2]{]\!]}{]\!]}{#1,#2}

\numberwithin{equation}{section}

\newif\ifbozza
\bozzafalse

\title{Well-posedness of monotone semilinear SPDEs with
  semimartingale noise}
  \author{Carlo Marinelli\thanks{Department of Mathematics, University
      College London, Gower Street, London WC1E 6BT, United
      Kingdom. URL: \texttt{http://goo.gl/4GKJP}} 
  \and Luca Scarpa\thanks{Faculty of Mathematics, University
      of Vienna, Oskar-Morgenstern-Platz 1, 1090 Vienna, Austria. 
      E-mail: \texttt{luca.scarpa@univie.ac.at}}}
\date{October 28, 2019}

\begin{document}
\maketitle

\begin{abstract}
  We prove existence and uniqueness of strong solutions for a class of
  semilinear stochastic evolution equations driven by general Hilbert
  space-valued semimartingales, with drift equal to the sum of a
  linear maximal monotone operator in variational form and of the
  superposition operator associated to a random time-dependent
  monotone function defined on the whole real line. Such a function is
  only assumed to satisfy a very mild symmetry-like condition, but its
  rate of growth towards infinity can be arbitrary. Moreover, the
  noise is of multiplicative type and can be path-dependent. The
  solution is obtained via a priori estimates on solutions to
  regularized equations, interpreted both as stochastic equations as
  well as deterministic equations with random coefficients, and
  ensuing compactness properties. A key role is played by an
  infinite-dimensional Doob-type inequality due to M\'etivier and
  Pellaumail.
  \medskip\par\noindent
  \emph{AMS Subject Classification:} 60H15, 47H06, 46N30.
  \medskip\par\noindent
  \emph{Key words and phrases:} stochastic evolution equations, singular drift, 
  semimartingale noise, variational approach, monotonicity methods
\end{abstract}


\section{Introduction}
\label{sec:intro}
Let us consider semilinear stochastic evolution equations of the type
\begin{equation}
  \label{eq:0}
  dX(t) + AX(t)\,dt + \beta(t,X(t))\,dt \ni B(t,X)\,dZ(t),
  \qquad X(0)=X_0,
\end{equation}
in $L^2(D)$, where $D$ is a smooth bounded domain of $\erre^n$. Here
$A$ is linear coercive maximal monotone operator on $L^2(D)$, $\beta$
is a random time-dependent maximal monotone graph everywhere defined
on the real line, $Z$ is a Hilbert space-valued semimartingale, and
the coefficient $B$ satisfies a suitable Lipschitz continuity
assumption (precise hypotheses on the data are given in
{\S}\ref{sec:ass} below).
Our main result is the existence and uniqueness of a strong solution
to \eqref{eq:0} (in the sense of Definition~\ref{def:sol} below), and
its continuous dependence on the initial datum in a suitable
topology.
Stochastic partial differential equations driven by semimartingales
arise naturally in several fields, such as physics, biology, and
finance, where a noise with possibly discontinuous trajectories can be
preferable, for modeling purposes, to the classical Wiener noise (see,
e.g., \cite{cm:EJP08,cm:MF10}). For further possible applications
where equations of the form \eqref{eq:0} are used we refer to
\cite{BrzLZ,LiuSte} and references therein.

Maximal monotone graphs such as $\beta$
arise naturally in the study of equations with non-linearities
associated to monotone discontinuous functions. In fact, it is well
known that every maximal monotone graph $\gamma$ in
$\erre \times \erre$ arises (in a unique way) from an increasing
function $\gamma_0:\erre \to \erre \cup \{+\infty\}$, setting
$\gamma(r):=[\gamma_0(r-),\gamma_0(r+)]$ for every $r \in \erre$,
i.e. by the procedure of ``filling the jumps''. Therefore our
treatment provides a notion of (strong) solution to stochastic
evolution equations of the type,
\[
  dX(t) + AX(t)\,dt + \beta_0(t,X(t))\,dt = B(t,X)\,dZ(t), \qquad
  X(0)=X_0,
\]
where $\beta_0(\omega,t,\cdot) \colon \erre \to \erre$ is an
increasing function, with possibly countably many discontinuities, and
with essentially no assumption on its rate of growth at
infinity. Stochastic evolution equations of this type are particularly
interesting as they cannot be handled using existing techniques, as
well as for their potential applications (equations with exponentially
growing drift appear, for instance, in mathematical models of
Euclidean quantum field theory -- see, e.g., \cite{Kawa}). In fact, to
the best of our knowledge, all results currently available in the
literature on stochastic equations with semimartingale noise are
obtained under assumptions on the coefficients that are too
restrictive to treat equation \eqref{eq:0}. In particular, after the
pioneering results by M\'etivier \cite{Met} for equations with bounded
$A$ and locally Lipschitz continuous drift and diffusion coefficients,
the first contribution to treat ``genuine'' stochastic evolution
equations (i.e., with $A$ unbounded) is probably \cite{Gyo-semimg},
where the well-posedness result in the variational setting for
equations with Wiener noise of \cite{KR-spde, Pard} is extended to the
case where the driving noise is a quasi left-continuous locally
square-integrable martingale, although under a rather restrictive
growth assumption on the (nonlinear) drift term. In particular,
semilinear equations such as \eqref{eq:0} can be treated with this
approach only if $\beta$ is Lipschitz continuous. More recently,
nonlinear equations in the variational setting driven by compensated
Poisson random measures have been considered, also under relaxed
monotonicity conditions, in \cite{BrzLZ}. Semilinear equations with
drift $A+\beta$, as in \eqref{eq:0}, can be treated within this
framework under polynomial growth assumptions on $\beta$ that depend
on the dimension of the domain $D \subset \erre^n$: the larger $n$ is,
the slower (polynomial) growth is allowed for $\beta$
(cf.~\cite{LiuRo} for a discussion of this issue). Our results do not
suffer of this drawback, as the growth rate of $\beta$ is not limited
in any way by the dimension $n$.
Multivalued stochastic equations with possibly c\`adl\`ag additive
noise have been studied also in \cite{GessTol}, under a linear growth
condition on the drift, so that semilinear equations such as
\eqref{eq:0} can be treated only if $\beta$ has at most linear growth.
Using semigroup methods, well-posedness for \eqref{eq:0} in the mild
sense is proved in \cite{cm:Ascona13,cm:EJP10}, under the assumptions
that $\beta$ grows polynomially and the noise is the sum of a Wiener
process and a compensated Poisson random measure (one should note,
however, that $A$ needs not admit a variational formulation).
The well-posedness result for \eqref{eq:0} obtained here should be
interesting also in the finite-dimensional setting, i.e. for
stochastic (ordinary) differential equations driven by
finite-dimensional semimartingales. In fact, apart of the classical
well-posedness results for equations with locally Lipschitz
coefficients (see, e.g., \cite{MetPell,Protter}), it seems that the
only work dealing with equations with monotone coefficients is
\cite{Jac:WP}, where, however, linear growth is required.

The strong solution to \eqref{eq:0} is constructed as limit of
solutions to approximating equations. In particular, replacing both
$A$ and $\beta$ with their Yosida approximations, one obtains a family
of approximating equations with bounded coefficients that admit
classical solutions in $L^2(D)$, thanks to results by M\'etivier and
Pellaumail (see \cite{Met,MetPell}). This double regularization is
necessary because, due to the general semimartingale noise, one cannot
simply regularize $\beta$ and rely on the classical variational theory
in \cite{Gyo-semimg,KR-spde,Pard}. Since we allow $\beta$ to be
random, care is needed to make sure that its Yosida approximation is
at least a progressively measurable function (see~{\S}\ref{sec:ass}
below for detail on this technical issue).
Furthermore, we first consider such regularized equations with
additive noise, i.e. with $B$ possibly random, but not dependent on
the unknown, and with the semimartingale $Z$ satisfying extra
integrability conditions that are removed in a second step.
Interpreting such approximating equations either as ``true''
stochastic equations or as deterministic evolution equations with
random coefficients (cf.~\cite{cm:ref,cm:AP18}), we obtain a priori
estimates for their solutions in various topologies. This idea has
already been used in \cite{cm:AP18} to deal with the well-posedness of
semilinear equations with singular drift and Wiener noise, and later
in \cite{cm:ref} to study regularity properties of their solutions.
The much more general assumptions on the noise in the present
situation give rise to several difficulties that require new ideas
with respect to \cite{cm:ref,cm:AP18}.  A fundamental tool is an
infinite-dimensional maximal inequality for stochastic integrals with
respect to semimartingales due to M\'etivier and Pellaumail
(see~\cite{Met,MetPell}).
These a priori estimates imply enough compactness to pass to the limit
in the regularized equations, thus solving a version of \eqref{eq:0}
with additive noise. The assumption that $\beta$ is everywhere defined
plays here a crucial role, as it allows to use weak compactness
techniques in $L^1$ spaces.
In order to treat the general case with multiplicative noise, we
proceed as follows: using localization techniques, we first show the
existence of strong solutions on closed stochastic intervals. This
technique also allows to remove the extra integrability assumption on
$Z$. Uniqueness of solutions on closed stochastic intervals implies
that such local solutions form a directed system, so that it is
natural to construct a maximal solution.  Finally, the linear growth
of $B$ is shown to imply that the maximal solution can be extended to
any compact time interval. One can also show that the solution depends
continuously on the initial datum in the sense of the topology of
uniform (in time) convergence in probability.

Several auxiliary results are needed to carry out the program outlined
above, some of which are interesting in their own right. For instance,
we prove a general version of It\^o's formula for the square of the
$L^2(D)$-norm in a variational setting with possibly singular terms.
This can be seen as an extension of the classical formulas by Pardoux,
Krylov, and Rozovski\u{\i} \cite{KR-spde,Pard}, as well as by Krylov
and Gy\H{o}ngy \cite{KryGyo2}, at least in the case where the
variational triple is Hilbertian. We shall investigate in more detail
It\^o-type formulas in (generalized) variational settings in a work in
preparation. We also give a characterization of weakly c\`adl\`ag
processes in terms of essential boundedness (in time) and a weak
c\`adl\`ag property in a larger space, extending the classical result
on weak continuity for vector-valued functions by Strauss (see
\cite{Strauss}).

The remaining text is organized as follows: in {\S}\ref{sec:ass} we
fix the notation, collect all standing assumptions, and discuss some
notable consequences thereof that are going to be used
extensively. The definition of strong solution, both in the global and
the local sense, and the statement of the main well-posedness result
are given in {\S}\ref{sec:res}.  In {\S}\ref{sec:prel} we recall some
elements of the above-mentioned approach by M\'etivier and Pellaumail
to stochastic integration with respect to semimartingales in Hilbert
space, centered around a fundamental stopped Doob-type inequality. We
also prove an extension to the c\`adl\`ag case of a classical
criterion for weak continuity of vector-valued function due to
Strauss, as well as a slight generalization of a classical criterion
for uniform integrability by de la Vall\'e-Poussin. In
{\S}\ref{sec:Ito} we prove an It\^o-type formula for the square of the
$L^2(D)$ norm of a process that can be decomposed into the sum of a
stochastic integral with respect to a (Hilbert-space-valued)
semimartingale and of a Lebesgue integral of a singular drift
term. This result is an essential tool to obtain, in
{\S}\ref{sec:add}, an auxiliary well-posedness result for a version of
\eqref{eq:0} with additive noise. Finally, the proof of the main
result is presented in {\S}\ref{sec:mult}.

\medskip

\noindent\textbf{Acknowledgment.} Large part of the work for this
paper was done during several stays of the first-named author at the
Interdisziplin\"ares Zentrum f\"ur Komplexe Systeme (IZKS),
Universit\"at Bonn, Germany, as guest of Prof.~S.~Albeverio. His kind
hospitality and the excellent working conditions at IZKS are
gratefully acknowledged. The second-named author
was funded by Vienna Science and Technology Fund (WWTF) 
through Project MA14-009.


\ifbozza\newpage\else\fi
\section{Assumptions and first consequences}
\label{sec:ass}
\subsection{Notation}
Every Banach space is intended as a real Banach space.
For any Banach spaces $E$ and $F$, we shall denote the Banach space of
continuous linear operators from $E$ to $F$ by $\cL(E,F)$, if endowed
with the operator norm, and by $\cL_s(E,F)$, if endowed with the
strong operator topology (i.e. with the topology of simple
convergence). If $E=F$, we shall just write $\cL(E)$ in place of
$\cL(E,E)$.
The usual Lebesgue-Bochner spaces of $E$-valued functions on a measure
space $(Y,\mathscr{A},m)$ will be denoted by $L^p(Y;E)$,
$p \in [0,\infty]$, where $L^0(Y;E)$ is endowed with the (metrizable)
topology of convergence in measure. The set of continuous functions
and of weakly continuous functions on $[0,T]$ with values in $E$ will
be denoted by $C([0,T];E)$ and $C_w([0,T];E)$,
respectively. Analogously, the symbols $D([0,T];E)$ and $D_w([0,T];E)$
stand for the corresponding spaces of c\`adl\`ag functions. A function
$f:Y \to \cL(E,F)$ will be called strongly measurable if it is the
limit in the norm topology of $\cL(E,F)$ of a sequence of elementary
functions. For every $f\in D([0,T];E)$ we shall use the symbol 
$f^*$ for $\sup_{t\in[0,T]}\norm{f(t)}_E$.

We shall denote by $D$ a smooth bounded domain of $\erre^n$, and by
$H$ the Hilbert space $L^2(D)$ with its usual scalar product
$\ip{\cdot}{\cdot}$ and norm $\norm{\cdot}$. 

All random elements will be defined on a fixed probability space
$(\Omega, \cF, \P)$ endowed with a filtration $(\cF_t)_{t\in\erre_+}$
satisfying the ``usual assumptions'' of right-continuity and
completeness. Identities and inequalities between random variables
will always be meant to hold $\P$-almost surely, unless otherwise
stated. Two (measurable) processes will be declared equal if they are
indistinguishable. By $Z$ we shall denote a fixed semimartingale
taking values in a (fixed) separable Hilbert space $K$. The standard
notation and terminology of stochastic calculus for semimartingales
will be used (see, e.g., \cite{Met}).

For any $a,b\in\erre$ we shall write $a\lesssim b$ to indicate that there 
exists a constant $c>0$ such that $a\leq cb$.

\subsection{Assumptions}
The following hypotheses will be in force throughout the paper. 

\medskip\par\noindent
\textbf{Assumption (A).}  We assume that $A \in \cL(V,V')$, where $V$
is a separable Hilbert space densely, continuously and compactly
embedded in $H$, and that there exists a constant $c>0$ such that
\[
  \ip{Au}{u} \geq c \norm{u}_V^2 \qquad \forall u \in V.
\]
We denote by $A_2$ the part of $A$ in $H$, i.e. the unbounded linear 
operator $(A_2, \dom(A_2))$ on $H$ defined as 
$A_2v:=Av$ for $v\in\dom(A_2):= \{v \in V: Av \in H\}$.  Furthermore, we assume that
there exists a sequence $(T_n)_{n\in\enne}$ of linear injective operators on
$L^1(D)$ such that, for every $n \in \enne$,
\begin{itemize}
\item[(a)] $T_n:L^1(D)\to L^1(D)$ is sub-Markovian, i.e., if
  $f \in L^1(D)$ with $0 \leq f\leq 1$ a.e.~in $D$, then
  $0 \leq T_nf \leq 1$ a.e.~in $D$;
\item[(b)] $T_n$ is ultracontractive, i.e
  $T_n\in \cL(L^1(D), L^\infty(D))$.
\end{itemize}
Moreover, denoting the restriction of $T_n$ to $H$ by the same symbol,
we assume that
\begin{itemize}
\item[(c)] $T_n \in \cL(H, V)$ for every $n \in \enne$ and it can be
  extended to a continuous linear operator on $V'$, still denoted by
  the same symbol;
\item[(d)] $T_n$ converges to the identity in $\cL_s(E)$, with
  $E \in \{L^1(D),V,H,V'\}$, as $n \to \infty$;
\item[(e)] $T_n(H)=T_m(H)$ for every $n,m\in\enne$.
\end{itemize}
Throughout the work, we shall denote by $V_0$ a Hilbert space
continuously embedded in $V \cap L^\infty(D)$ and dense in $V$. Thanks
to the assumptions on $(T_n)$ such a space always exists, for instance
setting $V_0 := T_{\bar{n}}(H)$, with $\bar{n}$ an arbitrary (but
fixed) natural number. Indeed, $V_0$ is independent of $\bar n$ thanks
to (e), so that for every $v\in V$ the sequence $(T_nv)_n\subset V_0$
converges to $v$ in $V$ thanks to (d). An arbitrary but fixed terminal
time will be denoted by $T$.

\medskip\par\noindent
\textbf{Assumption (J).} 
Let
$j:\Omega \times[0,T] \times \erre \to \erre_+$ be a function
satisfying the following conditions:
\begin{itemize}
\item[(a)] $j(\cdot,\cdot,x)$ is progressively measurable for all
  $x \in \erre$;
\item[(b)] $j(\omega,t,\cdot)$ is convex for every
  $(\omega,t) \in \Omega \times [0,T]$, with $j(\cdot,\cdot,0)=0$;
\item[(c)] one has
  \[
    \limsup_{|x| \to \infty} \frac{j(\omega,t,x)}{j(\omega,t,-x)} <
    \infty
  \]
  uniformly with respect to $(\omega,t) \in \Omega \times [0,T]$.
\end{itemize}
For every $(\omega,t) \in \Omega \times [0,T]$, the maximal monotone
graph $\beta(\omega,t,\cdot) \subset \erre^2$ is defined as the
subdifferential of $j(\omega,t,\cdot)$, i.e.
$y \in \beta(\omega,t,x)$ if and only if
\[
  j(\omega,t,x) + y(z-x) \leq j(\omega,t,z) \qquad \forall z \in \erre.
\]
Seeing the maximal monotone graph $\beta(\omega,t,\cdot)$ as a
multivalued map, condition (b) implies that
\begin{itemize}
\item[(d)] $\beta(\omega,t,\cdot)$ is everywhere defined for every
  $(\omega,t) \in \Omega \times [0,T]$.
\end{itemize}
We further assume that
\begin{itemize}
\item[(e)] $\beta(\omega,t,\cdot)$ is bounded on bounded sets
  uniformly with respect to $(\omega,t)$.
\end{itemize}  

\medskip\par\noindent
The forthcoming assumptions on the coefficient $B$ are formulated in
terms of control processes for semimartingales, whose definition is
given in {\S}\ref{ssec:int} below.
\medskip\par\noindent
\textbf{Assumption (B).} Let
$B:\Omega \times[0,T] \times D([0,T]; H) \to \cL(K,H)$
be a map satisfying the following conditions:
\begin{itemize}
\item[(a)] the process $B(\cdot,\cdot,u)$ is a strongly predictable
  $\cL(K,H)$-valued process for every adapted c\`adl\`ag $H$-valued
  process $u$;
\item[(b)] for every stopping time $\tau \leq T$, and for every adapted
  c\`adl\`ag $H$-valued processes $u$, $v$,
  \[
  u1_{\co{0}{\tau}} = v1_{\co{0}{\tau}} \qquad\text{implies} \qquad
  B(\cdot,u) 1_{\cc{0}{\tau}} = B(\cdot,v)1_{\cc{0}{\tau}};
  \]
\item[(c)] for every control process $C$ of $Z$ there exists an
  increasing, nonnegative, right-continuous, adapted process $L$ such
  that, for every $t \in \mathopen]0,T]$ and every adapted c\`adl\`ag $H$-valued
  processes $u$, $v$, one has
\begin{align*}
  \int_0^t \norm[\big]{B(s,u)-B(s,v)}^2_{\cL(K,H)}\,dC(s) \leq
  \int_0^t \sup_{r<s}\norm[\big]{u(r)-v(r)}^2\,dL(s),\\
  \int_0^t \norm[\big]{B(s,u)}^2_{\cL(K,H)}\,dC(s) \leq
  \int_0^t \Bigl( 1+\sup_{r<s}\norm[\big]{u(r)}^2 \Bigr)\,dL(s).
\end{align*}
\end{itemize}
Assumptions (a) and (b) are immediately satisfied if $B$ is of the form
$B(\omega,t,u)=\widetilde{B}(\omega,t,u(t-))$ for
all $u\in D([0,T]; H)$ and 
$(\omega,t) \in \Omega \times [0,T]$, with the convention 
$u(0-):=u(0)$, where
$\widetilde{B}: \Omega \times[0,T] \times H \to \cL(K,H)$ is strongly
measurable with respect to the product $\sigma$-algebra of the
predictable $\sigma$-algebra and of the Borel $\sigma$-algebra of
$H$. A more refined criterion can be found in
\cite[{\S\S}6.2--6.4]{MetPell}.

\medskip\par\noindent
Finally, the initial datum $X_0$ is an $H$-valued
$\cF_0$-measurable random variable.

\subsection{On assumptions (A) and (J)}
Assumptions (A) and (J) have important consequences that will be
extensively used in the sequel. The most important ones are collected
in this subsection.

The hypotheses on $V$ and $A$ ensure that $(V,H,V')$ is a Hilbertian
variational triple and that the operator $A$ is maximal monotone from
$V$ to $V'$. Moreover, as it follows by coercivity, linearity, and
monotonicity, $A$ is bijective form $V$ to $V'$. However, in
applications it is often necessary to consider only the weaker
coercivity on $A$
\[
  \ip{Au}{u} \geq c \norm{u}_V^2 - \delta \norm{u}^2
  \qquad \forall u \in V,
\]
with $\delta>0$ a constant. This case can be included in our analysis
by considering the operator $A+\delta I$ instead of $A$.

The hypotheses on $A$ are met by large classes of differential
operators (second order symmetric and non-symmetric divergence-form
operators, as well as the fractional Laplacian, for example) -- see,
e.g., \cite{cm:AP18} for a detailed list of concrete examples.

The standard example of a family of operators $(T_n)$ that can be
shown to satisfy conditions (a)--(d) above for large classes of
operators $A$ is $T_n := (I + (1/n)A)^{-m}$, with $m \in \enne$
sufficiently large. We refer again to, e.g., \cite{cm:AP18} for a
discussion of this issue. Moreover, note that for $T_n$ to belong to
$\cL(V')$ it suffices that the commutator
$R_n:=T_nA_2-AT_n:\dom(A_2) \to V'$ can be continuously extended to a
linear bounded operator from $V \to V'$. In fact, this allows to
extend $T_n$ to a linear bounded operator on $V'$ as follows: for any
$y \in V'$, by surjectivity of $A$ one has $y = Au$, with $u \in
V$. Setting $T_ny := R_nu + AT_nu \in V'$, in order to check that this
is well defined it is sufficient to prove that if $u \in V$ is such
that $Au=0$, then $R_nu + AT_nu=0$.  Let $u \in V$ be such that
$Au=0$. Then $Au \in H$, hence $u \in \dom(A_2)$, and
$0=Au=A_2u$. Since $T_n$ has already been defined on $H$, we have
$0=T_nA_2u=R_nu + AT_nu$.  Finally, we have
\begin{align*}
  \norm{T_ny}_{V'}
  &\leq\norm{R_nu}_{V'}+\norm{AT_nu}_{V'}
    \leq \norm{R_n}_{\cL(V,V')}\norm{u}_{V}
    + \norm{A}_{\cL(V,V')}\norm{T_n}_{\cL(V)}\norm{u}_V\\
  &\leq \left( \norm{R_n}_{\cL(V,V')}\norm{A^{-1}}_{\cL(V',V)}+
    \norm{A}_{\cL(V,V')}\norm{T_n}_{\cL(V)} \right) \norm{v}_{V'},
\end{align*}
so that $T_n:V'\to V'$ is also bounded.

The Banach-Steinhaus theorem implies that the sequence of linear
operators $(T_n)$ is bounded in $\cL(H)$, $\cL(V)$, and $\cL(V')$,
i.e.
\[
  \sup_{n\in\enne} \norm[\big]{T_n}_{\cL(H)}
  + \sup_{n\in\enne} \norm[\big]{T_n}_{\cL(V)} 
  +\sup_{n\in\enne} \norm[\big]{T_n}_{\cL(V')} < \infty.
\]
The continuity property of the adjoint family $(T_n^*)$ established
next plays an important role in the proof of the It\^o-type formula
for the square of the $H$-norm in \S\ref{sec:Ito}.
\begin{lemma}
  \label{lm:adj}
  The sequence of adjoint operators
  $(T^*_n)_{n\in\enne} \subset \cL(H)$ is contained in $\cL(V)$ and
  converges to the identity in $\cL_s(H)$.
\end{lemma}
\begin{proof}
  By the continuity of $(T_n)$ in $\cL_s(H)$ one has, for every $x$,
  $y \in H$,
  \[
    \ip{T_n^*x}{y} = \ip{x}{T_ny} \to \ip{x}{y},
  \]
  hence $T_n^*x$ converges weakly to $x$ in $H$ for every $x \in H$.
  Furthermore, for any $x \in V$ and $y \in H$, one has
  \[
    \ip{T_n^*x}{y} = \ip{x}{T_ny} \leq
    \norm[\big]{x}_V \norm[\big]{T_ny}_{V'}
    \leq N \norm[\big]{x}_V \norm[\big]{y}_{V'},
  \]
  where $N:=\sup_{n\in\enne}\norm{T_n}_{\cL(V')}$.  Since $H$ is
  densely and continuously embedded in $V'$, this readily implies that
  $T_n^*x \in V'' \simeq V$ and
  \[
    \norm[\big]{T_n^*x}_V \leq N \norm[\big]{x}_V
    \qquad \forall x \in V, \quad \forall n \in \enne.
  \]
  Since $V$ is reflexive, for any sequence $(n') \subset \enne$, there
  exist $z \in V$ and a subsequence $(n'') \subset (n')$, possibly
  depending on $x$, and such that $T_{n''}^*x$ converges weakly in $V$
  to $z$ as $n'' \to \infty$. Since $V$ is compactly embedded in $H$,
  $T_{n''}^*x$ converges strongly to $z$ in $H$. Recalling that
  $T_n^*x$ converges weakly to $x$ as $n \to \infty$, hence that so
  does $T_{n''}^*x$, we infer that $z=x$, i.e., $T_{n''}^*x$ converges
  strongly to $x$ in $H$. By a standard result of classical analysis,
  this yields the convergence of $T_n^*x$ to $x$ in $H$, that is,
  along the original sequence, which is independent of $x \in V$. The
  result can finally be extended to $x \in H$ by a density argument:
  let $(x_k) \subset V$ be a sequence converging to $x$ in $H$. The
  triangle inequality yields
  \begin{align*}
    \norm[\big]{T_n^*x-x}
    &\leq \norm[\big]{T_n^*x-T_n^*x_k}
    + \norm[\big]{T_n^*x_k-x_k} + \norm[\big]{x_k-x}\\
    &\leq \bigl( 1 + \sup\nolimits_n \norm[\big]{T_n}_{\cL(H)} \bigr)
    \norm[\big]{x_k-x} + \norm[\big]{T_n^*x_k-x_k},
  \end{align*}
  from which one easily concludes.
\end{proof}
\begin{rmk}
  In general, the adjunction map $T \mapsto T^*$ for linear bounded
  operators on a Hilbert space is continuous with respect to the
  uniform and the weak operator topology, but not with respect to the
  strong operator topology. The previous lemma thus identifies a
  (very!) special subset of linear bounded operators for which the
  adjunction map is continuous also with respect to the strong operator
  topology.
\end{rmk}

\medskip

Let us now discuss some consequences of assumption (J). For every
$(\omega,t) \in \Omega \times [0,T]$, let $j^*(\omega,t,\cdot)$ denote
the convex conjugate of $j(\omega,t,\cdot)$, defined as
\[
  j^*(\omega,t,y) = \sup_{x \in \erre} \bigl( xy - j(\omega,t,x) \bigr).
\]
The measurability and continuity hypotheses on $j$ imply that $j$ and
$j^*$ are normal integrands, or, equivalently, that their epigraphs
are progressively Effros-measurable (see, e.g., \cite{Hess:meas,
  Rock:int}). More precisely, let us recall that, given a function
$\phi: \Omega \times [0,T] \times \erre \to \erre$, its epigraph at
$(\omega,t) \in \Omega \times [0,T]$ is given by
\[
  \operatorname{epi} \phi(\omega,t) := \bigl\{ (x,y) \in \erre^2:\,
  \phi(\omega,t,x) \leq y \bigr\}.
\]
The progressive Effros-measurability of the epigraph of $\phi$ is then
defined as the progressive measurability of the set
\[
  \bigl\{ (\omega,t) \in \Omega \times[0,T]:\, \operatorname{epi}
  \phi(\omega,t) \cap E \neq \varnothing \bigr\}
\]
for every open $E \subset \erre^2$.

Moreover, if $j$ is a normal integrand, then $\beta$ is also
progressively Effros-measurable (see \emph{op.~cit}), which in turn
implies that the resolvent $(I+\lambda\beta)^{-1}$ and the Yosida
approximation $\beta_\lambda$ of $\beta$, both real-valued functions
on $\Omega \times [0,T] \times \erre$, are measurable with respect to
the product of the progressive $\sigma$-algebra and the Borel
$\sigma$-algebra (see, e,g., \cite[Proposition~3.12]{LiuSte}).

Assumption (c) can be interpreted by saying that, for any fixed
$(\omega,t)$, the rates of growth of $j$ at plus and minus infinity
are comparable. For instance, this is satisfied if $j(\omega,t,\cdot)$
is even for every $(\omega,t)$.

Assumption (d) implies that $j^*(\omega,t,\cdot)$ is superlinear at
infinity, uniformly with respect to $(\omega,t)$, i.e. that
\[
  \lim_{|y|\to+\infty}\frac{j^*(\omega,t,y)}{|y|} = +\infty
  \quad\text{uniformly in~$(\omega,t)\in\Omega\times[0,T]$}\,.
\]

Lastly, taking $z=0$ in the definition of $\beta$ as subdifferential
of $j$, assumption (e) implies that, for all
$(\omega,t) \in \Omega \times [0,T]$, $j(\omega,t,x) \leq yx$ for all
$y \in \beta(\omega,t,x)$, that is, $j(\omega,t,\cdot)$ is bounded on
bounded sets uniformly over $\Omega \times [0,T]$.

The above measurability conditions are obviously satisfied if $\beta$
is non-random and time-independent, i.e. if $\beta$ is an everywhere
defined maximal monotone graph in $\erre \times \erre$. Moreover, in
this case the convex function $j:\erre\to\erre_+$ such that
$\partial j=\beta$ and $j(0)=0$ is uniquely determined, and
$\dom(\beta)=\erre$ implies that $j^*$ is superlinear at infinity.

The boundedness assumption (e) is the natural generalization of the
analogous ones commonly used for time-dependent maximal monotone
graphs (see, e.g., \cite[p.~4]{Barbu:par}).

Note that all the conditions assumed to hold for every
$(\omega,t) \in \Omega \times [0,T]$ could have been assumed for
almost every $(\omega,t) \in \Omega \times [0,T]$ instead.  Indeed, in
such a case, if $E \subset \Omega \times[0,T]$ has measure $0$ and all
hypotheses hold outside $E$, then one can consider the restriction of
$j$ to the complement of $E$ instead of $j$.


\ifbozza\newpage\else\fi
\section{Main result}
\label{sec:res}
The concept of solution we are going to work with is as follows. We
recall that $T \in \erre_+$ is an arbitrary but fixed time horizon.
\begin{defi}
  \label{def:sol}
  Let $\tau \leq T$ be a stopping time. A strong solution on
  $[\![0,\tau]\!]$ to \eqref{eq:0} is a pair $(X,\xi)$, where $X$ is
  an adapted c\`adl\`ag $H$-valued process and $\xi$ is an adapted
  $L^1(D)$-valued process, such that
  \begin{itemize}
  \item[(a)] $\mathbbm{1}_{\cc{0}{\tau}} X \in L^1(0,T;V)$ and
    $\mathbbm{1}_{\cc{0}{\tau}} \xi \in L^1([0,T] \times D)$ $\P$-a.s., with
    $\xi \in \beta(\cdot,X)$ a.e. in $\co{0}{\tau} \times D$;
  \item[(b)] $\mathbbm{1}_{\cc{0}{\tau}}B(\cdot, X)$ is integrable with
    respect to $Z$;
  \item[(c)] one has, as an identity in $V' \cap L^1(D)$,
    \[
      X^\tau + \int_0^{\cdot\wedge\tau} AX(s)\,ds +
      \int_0^{\cdot\wedge\tau} \xi(s)\,ds = X_0
      + \bigl( \mathbbm{1}_{[\![0,\tau]\!]}B(\cdot,X) \bigr) \cdot Z.
    \]
  \end{itemize}
\end{defi}
\noindent%
A strong solution on $[0,T]$ will simply be called a strong solution.

\medskip

The main results of the paper are collected in the following
theorem. These ensure that \eqref{eq:0} admits a strong solution,
which is unique within a natural class of processes, and depends
continuously on the initial datum.
\begin{thm}
  \label{thm:1}
  Equation \eqref{eq:0} admits a strong solution $(X,\xi)$, with $X$
  optional, and it is the only one such that
  \[
    \sup_{t\leq T}\norm{X(t)}^2 + \int_0^T\norm{X(s)}_V^2\,ds +
    \int_0^T\!\!\int_D\xi(s)X(s)\,dx\,ds < \infty \qquad\P\text{-a.s.}
  \]
  Moreover, the solution map $X_0 \mapsto X$ is continuous from
  $L^0(\Omega;H)$ to $L^0(\Omega; D([0,T]; H) \cap L^2(0,T; V))$,
  where $D([0,T];H)$ is endowed with the topology generated by the
  supremum norm.
\end{thm}
\noindent
Note that since $\xi\in\beta(\cdot,X)$ we have $|\xi X|=\xi X=
j(\cdot,X)+j^*(\cdot,\xi)\geq0$,
so that Theorem~\ref{thm:1} ensures that
\[
  \xi X=j(\cdot,X) + j^*(\cdot,\xi) \in L^1((0,T)\times D) \quad\P\text{-a.s.}
\]


\ifbozza\newpage\else\fi
\section{Preliminaries and auxiliary results}
\label{sec:prel}
We recall those results from the approach to stochastic integration
developed by M\'etivier and Pellaumail that we need, referring to
\cite{Met,MetPell} for details. We also prove two additional lemmata
pertaining to this theory that are indispensable for the proofs in the
following sections.

Moreover, we provide a sufficient condition for a process to be weakly
c\`adl\`ag and a generalized version of the uniform integrability
criterion by de la Vall\'ee Poussin.

\subsection{Stochastic integration with respect to
  Hilbert-space-valued semimartingales}
\label{ssec:int}
Let $G$ be a separable Hilbert space. An $\cL(K,G)$-valued process $Y$
is elementary if there exist $n \in \enne$, sequences $(s_k)$,
$(t_k) \subseteq \erre_+$, $(F_k) \subset \cF$, and
$(u_k) \subset \cL(K,G)$, $k=1,\ldots,n$, with $s_k \leq t_k$ and
$F_k \in \cF_{s_k}$, such that
\[
  Y = \sum_{k=1}^n \mathbbm{1}_{\mathopen]s_k,t_k\mathclose] \times F_k}
  u_k.
\]
Then the stochastic integral of $Y$ with respect to $Z$ is defined as
\[
  \bigl( Y \cdot Z \bigr)_t := \sum_{k=1}^n \mathbbm{1}_{F_k}
  u_k \bigl( Z_{t_k \wedge t} - Z_{s_k \wedge t} \bigr)
  \qquad \forall t \in \erre_+.
\]
\begin{defi}
  A positive increasing adapted process $C$ is called a control
  process for $Z$ if, for every separable Hilbert space $G$, for every
  elementary $\cL(K,G)$-valued process $Y$, and for every stopping
  time $\tau$, one has
  \[
    \E\sup_{t<\tau} \norm[\big]{(Y \cdot Z)_t}_G^2
    \leq \E C_{\tau-} \int_{\mathopen]0,\tau\mathclose[} %
    \norm[\big]{Y(s)}^2_{\cL(K,G)}\,dC(s).
  \]
\end{defi}

It turns out that an adapted c\`adl\`ag $K$-valued process is a
semimartingale if and only if it admits a control process. In
particular, the set of control processes for a semimartingale $Z$,
that we shall denote by $\cC(Z)$, is not empty. One can also show (see
\cite[Theorems.~23.9--23.14]{Met}) that, writing $Z=M+V$, with $M$
locally square integrable local martingale and $V$ a finite-variation
process, a control process is given by
\[
  C=8\bigl( \ip{M}{M} + [\check{M},\check{M}] \bigr)
  + 2\bigl( 2\vee\abs{V} \bigr),
\]
where $\ip{M}{M}$ is the predictable quadratic variation of $M$,
$\abs{V}$ is the variation of $V$, and $[\check{M},\check{M}]$ is the
quadratic variation of the pure-jump martingale part of $M$, in the
sense of \cite[Definition~19.3]{Met}.

We need to introduce some notation: for any control process $C$ and
any strongly measurable adapted process $Y$ with values in
$\cL(K,G)$, let us define the process $\lambda^C(Y)$ as
\[
  \lambda_t^C(Y) := C_t \int_0^t \norm[\big]{Y(s)}^2_{\cL(K,G)}\,dC(s)
  \qquad \forall t \in \erre_+.
\]
For any stopping time $\tau$, let us define the measure $m^Z_\tau$ on
the predictable $\sigma$-algebra as
\[
  m^Z_\tau: P \longmapsto \E C_{\tau-}
  \bigl( \mathbbm{1}_P \cdot C \bigr)_{\tau-},
\]
and note that $m^Z_\tau$ is finite if
$\E\abs{C_{\tau-}}^2<\infty$. The space of strongly predictable
processes $Y$ with values in $\cL(K,G)$ such that
$\E\lambda^C_{\tau-}(Y)$ is finite coincides with the Bochner $L^2$
space with respect to the measure $m^Z_\tau$ and values in $\cL(K,G)$,
with norm
\[
  \norm[\big]{Y}_{L^2(m^Z_\tau)} = \bigl( \E\lambda^C_{\tau-}(Y) \bigr)^{1/2}
  = \Bigl( \E C_{\tau-} \bigl( \norm{Y}^2 \cdot C \bigr)_{\tau-} \Bigr)^{1/2},
\]
where the norm of $Y$ is taken in $\cL(K,G)$ (see \cite[\S~24.1,
Lemmata~1--3, and \S~26.1]{Met}). Denoting the Banach space of adapted
c\`adl\`ag processes $S$ with values in $G$ such that
$\E S^{*2}<\infty$ by $\mathbb{S}^2$, with norm
$\norm{S}_{S^2}:=(\E S^{*2})^{1/2}$, the inequality in the definition
of control process can thus be written as
\[
  \norm[\big]{(Y \cdot Z)^{\tau-}}_{\mathbb{S}^2}
  \leq \norm[\big]{Y}_{L^2(m^Z_\tau)}.
\]
The first step in the construction of the stochastic integral for more
general integrands is as follows: suppose that there exists a stopping
time $\tau$ such that $\E\abs{C_{\tau-}}^2 <\infty$, so that
$m^Z_\tau$ is a finite measure and the vector space of elementary
processes is dense in $L^2(m^Z_\tau)$. Then the mapping
$Y \mapsto (Y \cdot Z)^{\tau-}$, initially defined on elementary
processes, admits a unique extension to a linear continuous map from
$L^2(m^Z_\tau)$ to $\mathbb{S}^2$.
As a second step, assume that $C$ is a control process for $Z$ and $Y$
is a process with values in $\cL(K,G)$ such that the process
$\lambda^C(Y)$ is finite, and introduce the sequence
of stopping times $(\tau_n)$ defined as
\[
\tau_n := \inf \bigl\{ t \geq 0:\, C_t \wedge \lambda^C_t(Y) \geq n \bigr\},
\]
so that $\E\abs{C_{\tau_n-}}^2 <\infty$ as well as
$\E\lambda^C_{\tau_n-}(Y) < \infty$, i.e. $Y \in
L^2(m^Z_{\tau_n})$. Then, by the previous step, one has
$(Y \cdot Z)^{\tau_n-} \in \mathbb{S}^2$ for all $n \in \enne$. Since
$\tau_n$ increases to $\infty$ as $n \to \infty$ and it is not
difficult to show that $(Y \cdot Z)^{\tau_n-} = (Y \cdot Z)^{\tau_m-}$
on $\co{0}{\tau_n \wedge \tau_m}$ for all $n$, $m \in \enne$, one has
a well-defined process $Y \cdot Z$. One then shows that such a process
does not depend on the sequence $(\tau_n)$. However, it may still
depend on the control process $C$. A final step shows that if $Y$
admits two control processes $C_1$ and $C_2$ such that the processes
$\lambda^{C_1}(Y)$ and $\lambda^{C_2}(Y)$ are finite, then the
stochastic integrals constructed in the two possible ways coincide.
The following definition is therefore meaningful.
\begin{defi}
  A strongly predictable $\cL(K,G)$-valued process $Y$ is integrable
  with respect to $Z$ if there exists a control process $C$ for $Z$
  such that the process $\lambda^C(Y)$ is finite.
\end{defi}

We shall occasionally use the symbol $\cS_C(Z)$ to denote the set of
strongly predictable $\cL(K,G)$-valued processes $Y$ such that the
process $\lambda^C(Y)$ is finite.

Note that the construction of $Y \cdot Z$ implies that the inequality
in the definition of control processes can be extended as follows: for
every $C \in \cC(Z)$, $Y \in \cS_C(Z)$, and stopping time $\tau$, one has
\[
  \norm[\big]{(Y \cdot Z)^{\tau-}}_{\mathbb{S}^2} \leq
  \Bigl( \E C_{\tau-} \bigl( \norm{Y}^2 \cdot C \bigr)_{\tau-} \Bigr)^{1/2}.
\]

We shall need a further maximal inequality for stochastic integrals
with respect to a semimartingale, whose proof relies on the following
deep inequality (see~\cite[Lemma~1.3]{LeLePr}).
\begin{lemma}
  Let $X$ be a positive real-valued measurable process and $A$ an increasing
  predictable process such that, for every finite stopping time
  $\sigma$,
  \[
    \E \mathbbm{1}_{\{\sigma>0\}} X(\sigma) \leq %
    a \E \mathbbm{1}_{\{\sigma>0\}} A(\sigma)
  \]
  for a constant $a>0$. Then for every concave function
  $F:\erre_+ \to \erre$ and every finite stopping time $\tau$ one has
  \[
    \E \mathbbm{1}_{\{\tau>0\}} F(X(\tau)) \leq %
    (a+1) \E \mathbbm{1}_{\{\tau>0\}} F(A(\tau)).
  \]
\end{lemma}

Let $C$ be a control process for $Z$ and $Y \in \cS_C(Z)$, so that
\[
  \E \bigl( Y \cdot Z \bigr)_{\sigma-}^{*2} \leq
  \E C_{\sigma-} \bigl( \norm{Y}^2 \cdot C \bigr)_{\sigma-}.
\]
Since the process $C_{-} \bigl( \norm{Y}^2 \cdot C \bigr)_{-}$ is
left-continuous, hence predictable, the previous lemma yields, taking
$F(r)=\sqrt r$, $r\geq0$,
\[
\E \bigl( Y \cdot Z \bigr)_{\tau-}^{*} \leq 2
\E \bigl( C_{\tau-} \bigl( \norm{Y}^2 \cdot C \bigr)_{\tau-} \bigr)^{1/2}.
\]

\medskip

The following elementary lemma is essential in the last section.
\begin{lemma}
  \label{lm:stopped}
  Let $C$ be a control process for the semimartingale $Z$ and $\tau$ a
  stopping time. Then $C^{\tau-}$ is a control process for the
  semimartingale $Z^{\tau-}$.
\end{lemma}
\begin{proof}
  For every elementary $\cL(K,G)$-valued process $Y$ and every
  stopping time $\sigma$ one has
  $Y \cdot Z^{\tau-}=(Y \cdot Z)^{\tau-}$, hence also
  \[
    \bigl( Y \cdot Z^{\tau-} \bigr)^{*}_{\sigma-} =
    \bigl( Y \cdot Z \bigr)^*_{(\sigma \wedge \tau)-},
  \]
  which in turns implies
  \[
    \E\bigl( Y \cdot Z^{\tau-} \bigr)^{*2}_{\sigma-}
    = \E \bigl( Y \cdot Z \bigr)^{*2}_{(\sigma \wedge \tau)-}
      \leq \E C_{(\sigma \wedge \tau)-}%
      \bigl( \norm{Y}^2 \cdot C \bigr)_{(\sigma \wedge \tau)-},
  \]
  where $C_{(\sigma \wedge \tau)-} = C^{\tau-}_{\sigma-}$ and
  $\bigl( \norm{Y}^2 \cdot C \bigr)_{(\sigma \wedge \tau)-} = \bigl(
  \norm{Y}^2 \cdot C^{\tau-} \bigr)_{\sigma-}$.
\end{proof}

\medskip

We also recall the following version of the dominated convergence
theorem for stochastic integrals with respect to semimartingales
(cf.~\cite[Theorem~26.3]{Met}).
\begin{prop}
  \label{prop:dom1}
  Let $(X_n)_{n\in\enne}$, $X$ be predictable $\cL(K,H)$-valued
  processes such that $X_n \to X$ in $\cL(K,H)$ a.e. in
  $\Omega \times [0,T]$. If there exists a control process $C$ for $Z$
  and $\phi \in \cS_C(Z)$ such that
  \[
    \norm[\big]{X_n}_{\cL(K,H)}\leq \norm[\big]{\phi}_{\cL(K,H)}
    \qquad\forall n \in \enne,
  \]
  then $X_n \in \cS_C(Z)$ for every $n \in \enne$, $X \in \cS_C(Z)$, and 
  \[
    \bigl( X_n \cdot Z - X \cdot Z \bigr)^*_t \longto 0
  \]
  in probability for every $t \in [0,T]$.
\end{prop}

\medskip

Finally, we recall, for the reader's convenience, the following
stochastic version of Gronwall's lemma (cf.~\cite[Lemma~29.1]{Met}).
\begin{lemma}
  \label{lm:gron}
  Let $A$ be an adapted,
  right-continuous, increasing, positive process defined on a
  stochastic interval $[\![0,\tau[\![$, with
  $\ell:=\sup_{t<\tau}A(t)\in\erre_+$.  Let also $\phi$ be a real,
  increasing, adapted process such that, for every stopping time
  $\sigma\leq\tau$,
  \[
    \E\phi(\sigma-)\leq a + b \E\int_0^{\sigma-}\phi(s-)\,dA(s)
  \]
  for certain constants $a,b\in\erre$. Then,
  \[
    \E\phi(\tau-)\leq a\sum_{k=0}^{[2b\ell]}(2b\ell)^k.
  \]
\end{lemma}

\subsection{Weak right-continuity of vector-valued functions}
Throughout this section $E$ and $F$ denote two Banach spaces, with
$E$ reflexive, densely and continuously embedded in $F$.
A classical result by Strauss (see \cite{Strauss}) states that
\[
  L^\infty(0,T;E) \cap C_w([0,T];F) = C_w([0,T];E).
\]
We are going to show that the result continues to hold replacing the
spaces of weakly continuous functions by spaces of weakly c\`adl\`ag
functions.
\begin{lemma}
  \label{lm:strauss}
  One has
  \[
  L^\infty(0,T; E)\cap D_w([0,T]; F) = D_w([0,T]; E).
  \]
\end{lemma}
\begin{proof}
  The inclusion of the space on the right-hand side in the space on
  the left-hand side is evident. Let
  $u \in L^\infty(0,T; E) \cap D_w([0,T];F)$.
  Since $\{T\}$ is negligible with respect to the Lebesgue measure on
  $[0,T]$, it is not restrictive to suppose that $u(T)\in E$
  (otherwise, we shall modify the value of $u$ in $T$, 
  obtaining a version of $u$ which is still in 
  $L^\infty(0,T; E) \cap D_w([0,T];F)$). We first show that, in
  order for $u$ to belong to $D_w([0,T]; E)$, it suffices to prove
  that there exists a constant $M$ such that $\norm{u(t)}_{E} \leq M$
  for every $t \in [0,T]$.

  \noindent \textsc{Step 1.} Assuming that $u([0,T])$ is bounded in
  $E$, let $t \in [0,T)$ and $(t_n) \subset [t,T)$ be a sequence
  converging to $t$. Then $u(t_n) \to u(t)$ weakly in $F$ by
  assumption, and, since $E$ is reflexive, there exists a subsequence
  $(t_{n'})$ and $v \in E$ such that $u(t_{n'}) \to v$ weakly in
  $E$. Therefore $v=u(t)$ and $u(t_n) \to u(t)$ weakly in $E$,
  i.e. $u$ is weakly c\`ad with values in $E$. A completely analogous
  (in fact easier) argument shows that $u$ is also weakly l\`ag with values
  in $E$.
  
  \noindent\textsc{Step 2.} Let $(\rho_n)$ be a sequence of mollifiers
  in $\erre$ whose support is contained in
  $[-\frac{2}{n},0]$. Denoting the extension of $u$ to zero outside
  $[0,T]$ by the same symbol, it follows from
  $u \in L^\infty(\erre;E)$ that $u_n:=\rho_n*u \in C(\erre; E)$. In
  particular, Minkowski's inequality yields
  \[
    \norm[\big]{u_n(t)}_{E} \leq \int_\erre \abs[\big]{\rho_n(s)}
    \norm[\big]{u(t-s)}_{E}\,ds \leq \norm[\big]{u}_{L^\infty(0,T;E)}
    =: M
  \]
  for all $t \in \erre$ and $n \in \enne$.  Let $t_0 \in [0,T)$ be
  arbitrary but fixed. By reflexivity of $E$, there exist $v \in E$
  and a subsequence of $(u_n(t_0))$, denoted by the same symbol for
  simplicity, such that $u_n(t_0) \to v$ weakly in $E$.  Moreover, for
  any $\varphi \in F'$, 
  \[
    \ip{\varphi}{u_n}_F = \ip{\varphi}{\rho_n*u}_F
    = \rho_n*\ip{\varphi}{u}_F,
  \]  
  where $f:=\ip{\varphi}{u}_F \in D([0,T])$ by assumption. In
  particular, $f$ is right-continuous at $t_0$, i.e. for any
  $\delta>0$ there exists $N \in \enne$ such that
  $\abs{f(t_0-s)-f(t_0)}<\delta$ for all $s \in [-2/N,0]$. Since the
  support of $\rho_n$ is contained in $[-2/n,0]$, for $n>N$
  we have
  \begin{align*}
    \abs[\big]{\ip{\varphi}{u_n(t_0)}_{F} - \ip{\varphi}{u(t_0)}_{F}}
    &\leq \int_\erre \rho_n(s) \abs[\big]{f(t_0-s) - f(t_0)} \,ds\\
    &\leq \delta \int_\erre \rho_n(s) = \delta,
  \end{align*}
  i.e. $\ip{\varphi}{u_n(t)}_{F} \to \ip{\varphi}{u(t)}_{F}$ as
  $n \to \infty$. Since this holds for any $\varphi \in F'$, we infer
  that $u_n(t_0) \to u(t_0)$ weakly in $F$. Moreover, as $(u_n(t_0))$
  is bounded in $E$ and $E$ is reflexive, we easily deduce that
  $u_n(t_0) \to u(t_0)$ weakly in $E$, thus also, by weak lower
  semicontinuity of the norm, that
  \[
    \norm[\big]{u(t_0)}_E \leq \liminf_{n \to \infty} \norm[\big]{u_n(t_0)}_E
    \leq M.
  \]
  Since $t_0 \in [0,T)$ was arbitrary, this implies that
  $\norm{u(t)}_E \leq M$ for all $t \in [0,T)$. Moreover, 
  since $u(T)\in E$, we have that $u([0,T])$ is bounded
  in $E$, as required.
\end{proof}

\subsection{A criterion for uniform integrability}
We shall need a slightly generalized version of the de la
Vall\'ee-Poussin criterion for uniform integrability.  For the
purposes of this paragraph only, $(E,\mathscr{E},\mu)$ will denote a
finite measure space, and $m$ will stand for the product measure of $\P$,
the Lebesgue measure, and $\mu$ on $\Omega \times [0,T] \times E$. For
compactness of notation, we set
\[
  L^p(m) := L^p(\Omega \times [0,T] \times E,\cF \otimes
  \mathscr{B}([0,T]) \otimes \mathscr{E},m)
\]
for any $p \in [0,\infty]$.

\begin{lemma}
  \label{valle_p_time}
  Let $F:\Omega \times [0,T] \times \erre \to [0,+\infty]$ be proper,
  convex and lower semicontinuous in the third variable, measurable in
  the first two, and such that
  \[
    \lim_{\abs{x}\to +\infty} \frac{F(\omega,t,x)}{\abs{x}} = +\infty
    \qquad \text{uniformly in }(\omega,t)\in \Omega\times [0,T].
  \]
  If $\mathcal{G} \subseteq L^0(m)$ is such that there exists a
  constant $C$ for which
  \[
    \norm{F(\cdot,\cdot,g)}_{L^1(m)} < C
    \qquad \forall g \in \mathcal{G},
  \]
  then $\mathcal{G}$ is uniformly integrable in
  $\Omega \times [0,T] \times E$.
\end{lemma}
\begin{proof}
  We need to show that $\mathcal{G}$ is bounded in $L^1(m)$ and that
  for every $\varepsilon>0$ there exist $\delta$ such that, for any
  measurable set $A$ with $m(A)<\delta$, one has
  \[
    \int_A \abs{g}\,dm < \varepsilon \qquad \forall g \in \mathcal{G}.
  \]
  Let $M>0$ be a constant. By assumption there exists $R$ such that
  $x \in \erre$ with $\abs{x}>R$ implies $\abs{F(\omega,t,x)}>M|x|$
  for every $(\omega,t) \in \Omega \times[0,T]$. Then one has, for any
  $g\in \mathcal{G}$,
  \begin{align*}
    \int_A \abs{g}\,dm
    &= \int_{A \cap \{\abs{g}\leq R\}} \abs{g}\,dm
    +\int_{A \cap \{\abs{g}> R\}} \abs{g}\,dm\\
    & \leq R\, m(A)
    +\frac{1}{M} \int F(\cdot,\cdot,g)\,dm\\
    &\leq R\,m(A) + \frac{C}{M}.
  \end{align*}
  Choosing $A=\Omega \times [0,T] \times E$ it immediately follows
  that $\mathcal{G}$ is bounded in $L^1(m)$.  Moreover, for every
  $\varepsilon>0$ there exists $M$ such that
  $\frac{C}{M}<\frac\varepsilon2$, hence
  $\delta:=\frac\varepsilon{2R}$ satisfies the condition we are
  looking for.
\end{proof}

The same argument shows, keeping $\omega \in \Omega$ fixed, that if
there exists a finite positive random variable $C:\Omega \to \erre_+$
such that, for $\P$-a.e. $\omega \in \Omega$,
\[
  \norm{F(\omega,\cdot,g)}_{L^1([0,T] \times E)} < C(\omega)
  \qquad \forall g \in \mathcal{G},
\]
then $\mathcal{G}(\omega,\cdot)$ is uniformly integrable in
$(0,T)\times E$ for $\P$-a.e. $\omega \in \Omega$.


\ifbozza\newpage\else\fi
\section{The It\^o formula}
\label{sec:Ito}
In this section we prove an It\^o-type formula for the square of the
$H$-norm: this can be seen as an integration-by-parts formula in a
generalized setting.  We point out that the framework that we consider
here is is ``unusual'', as we work with processes with components in
$V'$ and $L^1(D)$ simultaneously, for which It\^o's formula is not
available using existing techniques.
Let us recall also that the quadratic variation of $Z$ is
defined as the process
\[
  [Z,Z] := \norm{Z}^2 - \norm{Z_0}^2 - 2 Z_- \cdot Z.
\]
In the sequel we shall denote $[0,T] \times D$ by $D_T$.
\begin{prop}
  \label{prop:Ito}
  Let $C$ be a control process for $Z$, $G \in \cS_C(Z)$, $Y_0\in
  L^0(\Omega,\cF_0,\P;H)$, and the adapted processes
  \begin{align*}
    Y &\in L^0(\Omega; L^\infty(0,T;H)) \cap L^0(\Omega; L^2(0,T; V))\\
    v &\in L^0(\Omega; L^2(0,T; V')),\\
    g &\in L^0(\Omega; L^1(0,T; L^1(D)))
  \end{align*}
  be such that 
  \begin{equation}
    \label{eq:Yg}
    Y + \int_0^\cdot v(s)\,ds + \int_0^\cdot g(s)\,ds
    = Y_0 + G \cdot Z.
  \end{equation}
  Furthermore, assume that there exists a real number $a>0$ such that
  \[
    j(\cdot,aY) + j^*(\cdot,ag) \in L^0(\Omega; L^1(D_T)).
  \]
  Then
  \begin{align*}
  &\frac12\norm{Y}^2 + \int_0^\cdot \ip{v(s)}{Y(s)}\,ds 
  +\int_0^\cdot\!\!\int_D g(s)Y(s)\,dx\,ds\\
    &\hspace{3em} = \frac12\norm{Y_0}^2
      + \frac12 \bigl[G\cdot Z,G \cdot Z \bigr]
      + (Y_-G) \cdot Z.
  \end{align*}
\end{prop}
\begin{proof}
  Let us first show that the stochastic integral $(Y_-G) \cdot Z$ is
  well defined: it follows from \eqref{eq:Yg} that $Y$ is strongly
  c\`adl\`ag in $V_0'$. Since
  $Y\in L^\infty(0,T; H)$, Lemma~\ref{lm:strauss} implies that $Y$ is
  weakly c\`adl\`ag in $H$, i.e. that, for any $h \in H$, $\ip{Y}{h}$
  is c\`adl\`ag, hence that $\ip{Y_-}{h}$ is left-continuous, in
  particular predictable, or, equivalently, that $Y_-$ is weakly
  predictable. However, since $H$ is separable, Pettis' theorem
  implies that $Y_-$ is predictable. Moreover, one has
  \begin{align*}
    \lambda^C_T(Y_-G)
    &= C(T) \bigl( \norm[\big]{Y_-G}^2_{\cL(K,\erre)} \cdot C \bigr)_T\\
    &\leq
    C(T) \sup_{t<T} \norm[\big]{Y(t)}^2
    \bigl( \norm[\big]{G}^2_{\cL(K,H)} \cdot C \bigr)_T
    = \sup_{t<T}\norm[\big]{Y(t)}^2 \lambda^C_T(G) < +\infty.
  \end{align*}
  Denoting the action of the operator $T_n$ by a superscript $n$, we
  have
  \[
    Y^n + \int_0^\cdot v^n(s)\,ds + \int_0^\cdot g^n(s)\,ds = Y^n_0 +
    G^n \cdot Z,
  \]
  as the Bochner integral as well as the stochastic integral commute
  with linear continuous operators. Since all integrands on the
  left-hand side are $H$-valued processes, the integration-by-parts
  formula for $H$-valued semimartingales yields (cf.~\cite[\S
  25]{Met})
  \begin{align*}
  &\frac12\norm{Y^n}^2 + \int_0^\cdot \ip{v^n(s)}{Y^n(s)}\,ds
  +\int_0^\cdot\!\!\int_Dg^n(s)Y^n(s)\,dx\,ds\\
    &\hspace{3em} = \frac12\norm{Y^n_0}^2
      + \frac12 \bigl[G^n\cdot Z,G^n \cdot Z\bigr]
      + (Y^n_- G^n) \cdot Z.
  \end{align*}
  We are now going to pass to the limit as $n \to \infty$ in this
  identity.  The continuity of $(T_n)$ in $\cL_s(H)$ immediately yields
  \begin{alignat*}{2}
    \norm{Y_0^n}^2 &\longto \norm{Y_0}^2, && \qquad \\
    \norm{Y^n(t)}^2 &\longto \norm{Y(t)}^2 && \qquad \forall t \in [0,T],\\
    G^n &\longto G && \qquad \text{in } \cL_s(K,H) \text{ a.e. in }
    \Omega \times [0,T].
  \end{alignat*}
  Similarly, since $(T_n)$ is also continuous in the strong operator
  topology of $V$, $V'$, and $L^1(D)$, the dominated convergence
  theorem readily implies that
  \begin{alignat*}{2}
    Y^n &\longto Y  && \quad \text{in } L^2(0,T; V),\\
    v^n &\longto v  && \quad \text{in } L^2(0,T; V'),\\
    g^n &\longto g  && \quad \text{in } L^1(D_T).
  \end{alignat*}
  In particular, passing to a subsequence if necessary, this implies
  that $g^nY^n\to gY$ almost everywhere in $D_T$. Therefore, if we
  show that $(g^nY^n)$ is uniformly integrable on $D_T$, we can
  conclude by Vitali's theorem that the latter convergence continues
  to hold also in $L^1(D_T)$. Thanks to the assumptions on
  the behavior at infinity of $j$, the sub-Markovianity of $T_n$, and
  the generalized Jensen inequality for positive operators
  (cf.~\cite{Haa07}), we have
  \begin{align*}
    \pm a^2g^nY^n \leq j(\cdot, \pm aY^n) + j^*(\cdot, ag^n)
    &\lesssim 1 + j(\cdot, aY^n) + j^*(\cdot, g^n)\\
    &\leq 1 + T_n\bigl( j(\cdot, aY) + j^*(\cdot, g) \bigr),
  \end{align*}
  where
  \[
    T_n\bigl( j(\cdot, aY) + j^*(\cdot, g) \bigr) \longto
    j(\cdot, aY) + j^*(\cdot, g) \qquad \text{in } L^1(D_T)
  \]
  as $n \to \infty$, because the right-hand side belongs to
  $L^1(D_T)$ a.s. by assumption. In particular,
  $T_n\bigl( j(\cdot, aY) + j^*(\cdot, g) \bigr)$ is uniformly
  integrable on $D_T$, and so is $(g^nY^n)$ by
  comparison. This implies, as explained above, that
  \[
    \int_0^\cdot\!\!\int_D g^n(s)Y^n(s)\,dx\,ds \longto
    \int_0^\cdot\!\!\int_D g(s)Y(s)\,dx\,ds.
  \]
  Let us now consider the quadratic variation term.  By definition we
  have
  \[
  \bigl[ G^n\cdot Z, G^n\cdot Z \bigr] 
  = \norm[\big]{G^n\cdot Z}^2 - 2((G^n\cdot Z)_-) \cdot (G^n\cdot Z),
  \]
  where the stochastic integral on the right-hand side can be written
  as $\tilde G^n\cdot Z$, with
  \[
  \tilde G^n:\Omega\times[0,T]\to\cL(K,\erre), \qquad 
  \tilde G^n(\omega,t)k:=\ip{(G^n\cdot Z)(\omega, s-)}{G^n(\omega, s)k}, 
  \quad k\in K.
  \]
  Noting that $G^n \cdot Z=T_n(G \cdot Z)$, it is immediate that
  $(G^n \cdot Z)_t \to (G\cdot Z)_t$ for all $t \in [0,T]$
  $\P$-a.s. as $n \to \infty$.
  Moreover, setting
  \[
  \tilde{G}:\Omega\times[0,T]\to \cL(K,\erre), \qquad
  \tilde{G}(\omega,s)k:=\ip{(G\cdot Z)(\omega, s-)}{G(\omega, s)k}, 
  \quad k\in K,
  \]
  one has
  \begin{align*}
  \tilde{G}^nk - \tilde{G}k 
  &= \ip[\big]{(G^n \cdot Z)_-}{G^nk} - \ip[\big]{(G\cdot Z)_-}{Gk}\\
  &= \ip[\big]{T_n^*T_n(G \cdot Z)_- - (G \cdot Z)_-}{Gk}\\
  &\leq \norm[\big]{T_n^*T_n(G\cdot Z)_- - (G\cdot Z)_-}
  \norm[\big]{G}_{\cL(K,H)} \norm{k}_K,
  \end{align*}
  where
  \begin{align*}
  \norm[\big]{T_n^*T_n (G \cdot Z)_- - (G \cdot Z)_-}
  &\leq \norm[\big]{T_n^*T_n (G\cdot Z)_- - T_n^* (G\cdot Z)_-}\\
  &\quad + \norm{T_n^*(G \cdot Z)_- - (G \cdot Z)_-}\\
  &\leq \sup_{n\in\enne} \norm[\big]{T_n^*}_{\cL(H)}
     \norm[\big]{T_n (G \cdot Z)_- - (G \cdot Z)_-}\\
  &\quad + \norm[\big]{T_n^* (G\cdot Z)_- - (G \cdot Z)_-},
  \end{align*}
  and the right-hand side converges to zero pointwise in time
  $\P$-a.s. because both $T_n$ and its adjoint converge to the
  identity operator in $\cL_s(H)$. Therefore $\tilde{G}^n$ converges
  to $\tilde{G}$ in $\cL(K,\erre)$ a.e. in $\Omega \times [0,T]$, and
  it follows by Proposition~\ref{prop:dom1} that
  \[
  \bigl[ G^n \cdot Z, G^n \cdot Z \bigr]_t \longto
  \bigl[ G \cdot Z, G \cdot Z]_t \qquad \forall t \in [0,T]
  \quad \P\text{-a.s.}
  \]
  Lastly, let us consider the convergence of the term
  $(Y_-^nG^n) \cdot Z$. Note that the
  $\cL(K,\erre)$-valued processes $Y^n_-G^n$ and $Y_-G$
  are defined as
  \[
  (Y^n_-G^n): k \longmapsto \ip[\big]{T_nY_-}{T_nGk}, \qquad
  (Y_-G) : k \longmapsto \ip[\big]{Y_-}{Gk},
  \]
  so that
  \begin{align*}
    \abs[\big]{(Y^n_-G^n-Y_-G)k}
    &= \ip[\big]{T_nY_-}{T_nGk} - \ip[\big]{Y_-}{Gk}\\
    &= \ip[\big]{T_n^*T_nY_- - Y_-}{Gk}\\
    &\leq \norm[\big]{T_n^*T_nY_- - Y_-}\norm[\big]{G}_{\cL(K,H)} \norm{k}_K,
  \end{align*}
  which in turn yields
  \[
    \norm[\big]{(Y^n_-G^n-Y_-G)}_{\cL(K,\erre)} \leq
    \norm[\big]{G}_{\cL(K,H)} \norm[\big]{T_n^*T_nY_- - Y_-}
    \qquad \text{a.e. in } \Omega \times (0,T].
  \]
  Recalling that $(T_n)_n$ is uniformly bounded in $\cL(H)$, hence so
  is $(T_n^*)_n$, it follows that
  \begin{align*}
    \norm[\big]{T_n^*T_nY_- - Y_-}
    &\leq \norm[\big]{T_n^*T_nY_- - T_n^*Y_-}
      + \norm[\big]{T_n^*Y_- - Y_-}\\
    &\leq \sup_{n\in\enne} \norm[\big]{T_n^*}_{\cL(H)}
      \norm[\big]{T_nY_- - Y_-} + \norm[\big]{T_n^*Y_- - Y_-},
  \end{align*}
  where the right-hand side converges to zero a.e. in
  $\Omega \times (0,T]$ thanks to the assumptions on $(T_n)$ and to
  Lemma~\ref{lm:adj}. Therefore $Y^n_-G^n$ converges to $Y_-G$ in
  $\cL(K,\erre)$ a.e. in $\Omega \times (0,T]$, so that
  Proposition~\ref{prop:dom1} allows us to conclude that
  $(Y^n_-G^n) \cdot Z$ converges to $(Y_-G) \cdot Z$ in probability
  uniformly in time.
\end{proof}


\ifbozza\newpage\else\fi
\section{Well-posedness with additive noise}
\label{sec:add}
The goal of this section is to establish a well-posedness result for
the following version of \eqref{eq:0} with additive noise:
\begin{equation}
  \label{eq:add}
  dX(t) + AX(t)\,dt + \beta(t,X(t))\,dt \ni G(t)\,dZ(t), \qquad X(0)=X_0,
\end{equation}
where $G$ is a strongly predictable $\cL(K,H)$-valued process
integrable with respect to $Z$. This is an essential step towards the
proof of the main results in the next section.

We begin with an existence result.
\begin{thm}
  \label{thm:add}
  Let $C$ be a control process for $Z$ and $G \in \cS_C(Z)$ such that
  $\E\lambda_{T-}^C(G) < \infty$ and assume that
  $X_0 \in L^2(\Omega;H)$.  Then \eqref{eq:add} admits a strong
  solution.
\end{thm}

The main idea of the proof is to regularize both $A$ and $\beta$ in
\eqref{eq:add}, so that the regularized equation admits a (unique)
strong solution in the classical sense, to obtain uniform estimates on
such solutions, and finally to pass to the limit using compactness and
monotonicity arguments.

For any $\lambda \in \mathopen]0,1\mathclose[$, let
$\beta_\lambda: \Omega \times[0,T] \times \erre\to \erre$ and
$A_\lambda \in \cL(H)$ be the Yosida approximations of
$r \mapsto \beta(\cdot,\cdot,r)$ and of $A_2$, respectively
(see \cite{Barbu:type} for references). Recall
that $A_2$ denotes the part of $A$ in $H$ and that, setting
$J_\lambda:=(I+\lambda A_2)^{-1}$, by definition of $A_\lambda$ we
have that $A_\lambda=AJ_\lambda$.

Let us consider the regularized equation
\begin{equation}
  \label{eq:reg}
  dX_\lambda(t) + A_\lambda X_\lambda(t)\,dt +
  \beta_\lambda(t,X_\lambda(t))\,dt = G(t)\,dZ(t),
  \qquad X_\lambda(0) = X_0.
\end{equation}
Since $A_\lambda + \beta_\lambda$ is Lipschitz continuous (uniformly
over $\Omega \times [0,T]$), the equation admits a unique strong
solution $X_\lambda$ in the classical sense, i.e. $X_\lambda$ is an
adapted c\`adl\`ag $H$-valued process, with
\[
  \E\sup_{t\leq T}\norm{X_\lambda(t)}^2<+\infty,
\]
such that 
\[
  X_\lambda + \int_0^\cdot A_\lambda X_\lambda(s)\,ds
  + \int_0^\cdot \beta_\lambda(s,X_\lambda(s))\,ds =
  X_0 + G\cdot Z
\]
(see \cite[Thm.~34.7--35.2]{Met}).

We are now going to establish a priori estimates on $(X_\lambda)$ and
functionals thereof.
\begin{lemma}
  \label{lm:est1}
  There exists a constant $N>0$ such that, for every
  $\lambda\in\mathopen]0,1\mathclose[$,
  \[
    \E\sup_{t<T} \norm{X_\lambda(t)}^2
    + \E \norm{J_\lambda X_\lambda}^2_{L^2(0,T; V)}
    + \E \norm{\beta_\lambda(\cdot, X_\lambda)X_\lambda}_{L^1(D_T)}
    <N.
  \]
\end{lemma}
\begin{proof}
  The integration-by-parts formula for $H$-valued processes yields
  \begin{align*}
    &\frac{1}{2}\norm{X_\lambda}^2
      + \int_0^\cdot \ip[\big]{A_\lambda X_\lambda(s)}{X_\lambda(s)}\,ds 
      + \int_0^\cdot\!\!\int_D\beta(s,X_\lambda(s))X_\lambda(s)\,ds\\
    &\hspace{3em} = \frac{1}{2}\norm{X_0}^2 
      + \frac{1}{2}\bigl[G\cdot Z,G \cdot Z\bigr]
      + ({X_\lambda}_- G) \cdot Z,
  \end{align*}
  where ${X_\lambda}_-$ denotes the process $(X_\lambda)_-$.  Taking
  the supremum in time over $[0,T\mathclose[$, recalling the identity
  \[
    \ip[\big]{A_\lambda X_\lambda}{X_\lambda} =
    \ip[\big]{AJ_\lambda X_\lambda}{J_\lambda X_\lambda}
    + \lambda \norm[\big]{A_\lambda X_\lambda}^2,
  \]
  one has, by coercivity of $A$,
  \begin{align*}
    &\bigl( X_\lambda \bigr)_{T-}^{*2}
      + 2c \int_0^T \norm[\big]{J_\lambda X_\lambda(s)}_V^2\,ds
      + 2\int_0^T\!\!\int_D\beta_\lambda(s,X_\lambda(s))X_\lambda(s)\,dx\,ds\\
    &\hspace{3em} \leq \norm{X_0}^2 + \bigl[G\cdot Z,G \cdot Z\bigr]_{T-}
      + 2 \bigl( ({X_\lambda}_- G) \cdot Z \bigr)^*_{T-}.
  \end{align*}
  We are going to estimate the last two terms on the right-hand side
  of the last inequality. By definition of quadratic variation we have
  \begin{align*}
    [G\cdot Z,G \cdot Z]
    &= \norm{G \cdot Z}^2 - 2(G \cdot Z)_- \cdot (G\cdot Z)\\
    &= \norm{G\cdot Z}^2 - 2 \tilde{G} \cdot Z,
  \end{align*}
  where $\tilde{G}: \Omega \times[0,T] \to \cL(K,\erre) \simeq K$ is
  defined as $K \ni k \mapsto \ip[\big]{(G \cdot Z)_-}{Gk}$. By definition of
  control process and by the second inequality for stochastic
  integrals in \S\ref{ssec:int} we thus have 
  \begin{align*}
    \E[G\cdot Z,G \cdot Z]_{T-}
    &\leq \E \bigl( G \cdot Z \bigr)_{T-}^{*2}
    + 2 \E \bigl( \tilde{G} \cdot Z \bigr)_{T-}^*\\
    &\leq  \E \lambda_{T-}^C(G)
    + 4 \E \lambda^C_{T-}(\tilde{G})^{1/2},
  \end{align*}
  where, by elementary inequalities,
  \begin{align*}
    \E \lambda^C_{T-}(\tilde{G})^{1/2}
    &= \E C(T-)^{1/2} \bigl( \norm{\tilde{G}}_K^2 \cdot C \bigr)_{T-}^{1/2}\\
    &\leq \E C(T-)^{1/2} \biggl(%
      \int_0^{T-} \norm{(G \cdot Z)_-}^2 \norm{G}^2_{\cL(K,H)}
      \,dC \biggr)^{1/2}\\
    &\leq \E \lambda^C_{T-}(G)^{1/2} \, \bigl( G \cdot Z \bigr)^*_{T-}\\
    &\leq \frac12 \E \lambda^C_{T-}(G)
      + \frac12 \E \bigl( G \cdot Z \bigr)_{T-}^{*2}
      \leq \E \lambda^C_{T-}(G),
  \end{align*}
  so that $\E [G \cdot Z,G \cdot Z]_{T-} \leq 5 \E \lambda^C_{T-}(G)$.
  Similarly, one has
  \begin{align*}
    \E \bigl( ({X_\lambda}_- G) \cdot Z \bigr)^*_{T-}
    &\leq 2 \E C(T-)^{1/2} \biggl(%
    \int_0^{T-} \norm{X_\lambda}^2 \norm{G}_{\cL(K,H)}^2\,dC \biggr)^{1/2}\\
    &\leq 2 \E \lambda^C_{T-}(G)^{1/2} \, \bigl( X_\lambda \bigr)^*_{T-}\\
    &\leq \frac14 \E \bigl( X_\lambda \bigr)^{*2}_{T-} 
    + 4 \E \lambda^C_{T-}(G),
  \end{align*}
  therefore also
  \begin{align*}
    &\E \bigl( X_\lambda \bigr)^{*2}_{T-}%
    + \E\int_0^T \norm[\big]{J_\lambda X_\lambda(s)}_V^2 \,ds
    + \E\int_0^T\!\!\int_D \beta_\lambda(s,X_\lambda(s))X_\lambda(s)\,dx\,ds\\
    &\hspace{3em} \lesssim \E \norm{X_0}^2 + \E \lambda^C_{T-}(G),
  \end{align*}
  uniformly over $\lambda \in \mathopen]0,1\mathclose[$, as the
  implicit constant depends only on $c$, the coercivity constant of
  $A$. We conclude noting that $\beta_\lambda(\cdot, X_\lambda)X_\lambda\geq 0$
  by monotonicity of $\beta_\lambda$.
\end{proof}

We are going to establish an existence and uniqueness result for
\eqref{eq:add} under the additional assumption that
\begin{equation}
  \label{eq:G0}
  G: \Omega \times [0,T] \to \cL(K,V_0).
\end{equation}
This is only a technical ``temporary'' assumption that will be
dispensed of in the proof of Theorem~\ref{thm:add}.
\begin{prop}
  \label{prop:aux}
  Assume that the hypotheses of Theorem~\ref{thm:add} hold and that
  $G$ satisfies \eqref{eq:G0}. Then \eqref{eq:add} admits a unique
  strong solution.
\end{prop}
For the proof we need further a priori estimates on the solution to
the regularized equation \eqref{eq:reg}.
\begin{lemma}
  \label{lm:est2}
  Let $G$ satisfy \eqref{eq:G0}. There exists $\Omega' \in \cF$ with
  $\P(\Omega')=1$ such that, for every $\omega \in \Omega'$, the
  following properties hold:
  \begin{itemize}
  \item[(a)] $\bigl(X_\lambda(\omega)\bigr)$ is bounded in
    $L^\infty(0,T;H)$;
  \item[(b)] $\bigl(J_\lambda X_\lambda(\omega)\bigr)$ is bounded in
    $L^2(0,T;V)$;
  \item[(c)] $\bigl(\lambda^{1/2} A_\lambda X_\lambda(\omega)\bigr)$
    is bounded in $L^2(0,T;H)$;
  \item[(d)]
    $\bigl(\beta_\lambda(\cdot,X_\lambda(\omega))X_\lambda(\omega)\bigr)$
    is bounded in $L^1([0,T] \times D)$.
  \end{itemize}
\end{lemma}
\begin{proof}
  Thanks to assumption \eqref{eq:G0}, there exists $\Omega' \in \cF$,
  with $\P(\Omega')=1$, such that
  \[
    (G \cdot Z)(\omega) \in L^\infty(0,T;V_0) \qquad \forall
    \omega \in \Omega'.
  \]
  Let $\omega \in \Omega'$ be arbitrary but fixed, so that indication
  of the explicit dependence on $\omega$ of the various processes
  involved will be suppressed for compactness of notation.  By
  inspection of \eqref{eq:reg} it follows that
  $X_\lambda- G \cdot Z \in H^1(0,T;V')$, so that we can write
  \[
  \frac{d}{dt}(X_\lambda-G\cdot Z) + A_\lambda X_\lambda +
  \beta_\lambda(\cdot, X_\lambda)= 0
  \]
  as an identity in $V'$ which holds for a.a.
  $t \in \mathopen]0,T\mathclose[$. The (deterministic)
  integration-by-parts formula then yields
  \begin{align*}
    &\frac12 \norm[\big]{X_\lambda - G \cdot Z}^2
      + \int_0^\cdot \ip[\big]{A_\lambda X_\lambda(s)}%
      {X_\lambda(s)-G \cdot Z(s)}\,ds\\
    &\hspace{3em} + \int_0^\cdot\!\!\int_D \beta_\lambda(s,X_\lambda(s))%
      (X_\lambda(s)-G \cdot Z(s))\,dx\,ds = \frac12 \norm{X_0}^2,
  \end{align*}
  where (i) by the triangle inequality and the elementary inequality
  $(a-b)^2/2 \geq \frac14a^2-\frac12b^2$, $a$, $b \in \erre$, one has
  \[
    \frac12 \norm[\big]{X_\lambda - G \cdot Z}^2 \geq
    \frac12 \bigl( \norm{X_\lambda} - \norm{G \cdot Z} \bigr)^2
    \geq \frac14\norm{X_\lambda}^2 - \frac12\norm{G \cdot Z}^2;
  \]
  (ii) one has, for any $h \in H$,
  $\ip{A_\lambda h}{h} = \ip{AJ_\lambda h}{J_\lambda h} +
  \lambda\norm{A_\lambda h}^2$, so that, by coercivity of $A$ and
  Young's inequality in the form
  $ab \leq \varepsilon a^2 + b^2/\varepsilon$, $a,b \in \erre$,
  $\varepsilon>0$, it follows that
  \begin{align*}
    \ip[\big]{A_\lambda X_\lambda}{X_\lambda - G \cdot Z}
    &= \ip[\big]{A_\lambda X_\lambda}{X_\lambda}
      - \ip[\big]{AJ_\lambda X_\lambda}{G \cdot Z}\\
    &\geq c \norm[\big]{J_\lambda X_\lambda}_V^2
      + \lambda \norm[\big]{A_\lambda X_\lambda}^2\\
    &\quad - \varepsilon \norm[\big]{A}^2_{\cL(V,V')}%
      \norm[\big]{J_\lambda X_\lambda}_V^2%
      + \frac{1}{\varepsilon} \norm[\big]{G \cdot Z}_{V}^2;
  \end{align*}
  (iii) one has, for any $x \in \erre$, slightly simplifying notation,
  \begin{align*}
    \beta_\lambda(x)x
    &= \beta_\lambda(x) (I+\lambda\beta)^{-1}(x)
      + \beta_\lambda(x) \bigl(x - (I+\lambda\beta)^{-1}(x) \bigr)\\
    &= \beta_\lambda(x) (I+\lambda\beta)^{-1}(x)
      + \lambda \abs[\big]{\beta_\lambda(x)}^2,
  \end{align*}
  hence also, recalling that
  $\beta_\lambda \in \beta \circ (I+\lambda\beta)^{-1}$ and that, for
  any $a, b \in \erre$, $ab = j(a) + j^*(b)$ if and only if
  $b \in \partial j(a) = \beta(a)$,
  \[
    \beta_\lambda(X_\lambda)X_\lambda \geq
    j\bigl((I+\lambda\beta)^{-1}(X_\lambda\bigr)
    + j^*\bigl(\beta_\lambda(X_\lambda)\bigr)
    \geq j^*\bigl(\beta_\lambda(X_\lambda)\bigr);
  \]
  (iv) Young's inequality in the form
  \[
    ab \leq j^*(\varepsilon a) + j(b/\varepsilon)
    \leq \varepsilon j^*(a) + j(b/\varepsilon),
    \qquad  \,a,b \in \erre, \; 0<\varepsilon<1,
  \]
  implies
  \[
    - \beta_\lambda(\cdot,X_\lambda) (G \cdot Z) \geq
    - \varepsilon j^*\bigl(\cdot,\beta_\lambda(\cdot,X_\lambda)\bigr)
    - j\bigl(\cdot,(G \cdot Z)/\varepsilon\bigr).
  \]
  Choosing $\varepsilon<1$, it follows from (i)--(iv) that
  \begin{align*}
    &\frac14\norm[\big]{X_\lambda}^2
      + c \int_0^\cdot \norm[\big]{J_\lambda X_\lambda(s)}_V^2\,ds
      + \lambda \int_0^\cdot \norm[\big]{A_\lambda X_\lambda(s)}^2\,ds\\
    &\hspace{3em} + \int_0^\cdot\!\!\int_D
      j^*\bigl(s,\beta_\lambda(s,X_\lambda(s))\bigr) \,dx\,ds\\
    &\leq \frac12\norm{X_0}^2 + \frac12\norm[\big]{G\cdot Z}^2\\
    &\hspace{3em}  + \varepsilon \norm[\big]{A}_{\cL(V,V')}
      \int_0^\cdot \norm[\big]{J_\lambda X_\lambda(s)}^2_V\,ds
      + \frac{1}{\varepsilon} \int_0^\cdot \norm[\big]{(G \cdot Z)_s}^2_V\,ds\\
    &\hspace{3em} + \varepsilon \int_0^\cdot\!\!\int_D
      j^*\bigl(s, \beta_\lambda(s,X_\lambda(s))\bigr)\,dx\,ds
      + \int_0^\cdot\!\!\int_D 
      j\bigl(s, (G \cdot Z)_s/\varepsilon\bigr)\,dx\,ds.
  \end{align*}
  First rearranging terms and choosing $\varepsilon$ sufficiently
  small, then taking the essential supremum in time, one gets
  \begin{align*}
    &\norm[\big]{X_\lambda}^2_{L^\infty(0,T;H)}
      + \norm[\big]{J_\lambda X_\lambda}^2_{L^2(0,T;V)}
      + \lambda \norm[\big]{A_\lambda X_\lambda}^2_{L^2(0,T; H)}
    +\norm[\big]{j^*(\cdot, \beta_\lambda(\cdot,X_\lambda))}_{L^1(D_T)}\\
    &\hspace{3em} \lesssim \norm[\big]{X_0}^2
      + \norm[\big]{G\cdot Z}_{L^2(0,T;V)}^2
      + \int_{D_T} j\bigl(s, (G \cdot Z)_s/\varepsilon\bigr)\,dx\,ds,
  \end{align*}
  where the right-hand side is finite because
  $G \cdot Z \in L^\infty(0,T;V_0)$. In fact, recalling that $V_0$ is
  continuously embedded in $V$, this immediately implies that
  $G \cdot Z \in L^2(0,T;V)$; moreover, there exists
  $D'_T \subset D_T$, with $D_T \setminus D'_T$ of measure zero, such
  that the restriction of $G \cdot Z$ to $D'_T$ is bounded. The
  finiteness of the last term on the right-hand side then follows by
  the boundedness on bounded sets of $y \mapsto j(\omega,t,y)$
  uniformly over $(\omega,t) \in \Omega \times [0,T]$.
\end{proof}

The pathwise boundedness properties just proved entail several
compactness properties in suitable topologies.
\begin{lemma}
  \label{lm:cvg}
  Let $G$ satisfy \eqref{eq:G0}. There exists $\Omega'\in\cF$ with
  $\P(\Omega')=1$ such that, for every $\omega \in \Omega'$, there
  exist a subsequence $\lambda'=\lambda'(\omega)$ of $\lambda$ and
  \[
    X(\omega) \in L^\infty(0,T;H) \cap L^2(0,T;V),
    \qquad \xi(\omega) \in L^1([0,T] \times D)
  \]
  such that 
  \begin{alignat*}{3}
    X_{\lambda'}(\omega,\cdot) &\longto X(\omega,\cdot) \quad
    && \text{weakly* in } L^{\infty}(0,T;H),\\
    X_{\lambda'}(\omega,\cdot) &\longto X(\omega,\cdot) \quad
    && \text{in } L^2(0,T;H),\\
    J_{\lambda'} X_{\lambda'}(\omega,\cdot) &\longto X(\omega,\cdot)  \quad
    && \text{weakly in } L^{2}(0,T;V),\\
    \beta_{\lambda'}(\cdot,X_{\lambda'}(\omega,\cdot)) &\longto \xi(\omega,\cdot)
    \qquad && \text{weakly in } L^1([0,T] \times D).
  \end{alignat*}
\end{lemma}
\begin{proof}
  Let $\Omega'$ be as in Lemma~\ref{lm:est2} and $\omega \in \Omega'$
  arbitrary but fixed (whose indication will still be omitted).  Since
  $(X_\lambda)$ is bounded in $L^\infty(0,T;H)$, hence also in
  $L^2(0,T;H)$, there exist $X \in L^\infty(0,T;H)$ and a subsequence
  $\lambda'$, depending on $\omega$, such that $X_{\lambda'}$
  converges weakly* to $X$ in $L^\infty(0,T;H)$ and weakly in
  $L^2(0,T;H)$. The boundedness of $(J_\lambda X_\lambda)$ in
  $L^2(0,T;V)$ implies that there exists $\bar{X} \in L^2(0,T;V)$ such
  that $J_{\lambda'}X_{\lambda'}$ converges weakly to $\bar{X}$ in
  $L^2(0,T;V)$. Boundedness of $(\sqrt{\lambda}A_\lambda X_\lambda)$
  in $L^2(0,T;H)$ implies that $\lambda A_\lambda X_\lambda$ converges
  to zero in $L^2(0,T;H)$. Writing
  \[
    J_\lambda X_\lambda = X_\lambda - \lambda A_\lambda X_\lambda,
  \]
  one immediately infers that $J_{\lambda'}X_{\lambda'}$ converges
  weakly to $X$ in $L^2(0,T;H)$. Since it also converges weakly to
  $\bar{X}$ in $L^2(0,T;V)$, it follows that $\bar{X}=X$.
  
  The same argument used in part (iii) of the proof of
  Lemma~\ref{lm:est2} yields
  \[
    j^*(t,\beta_\lambda(t,X_\lambda)) \leq
    \beta_{\lambda}(t,X_\lambda)X_\lambda,
  \]
  where the right-hand side, as a family indexed by $\lambda$, is
  bounded in $L^1(D_T)$. The generalized de la Vall\'ee-Poussin
  criterion of Lemma~\ref{valle_p_time} then ensures that
  $(\beta_\lambda(\cdot,X_\lambda))$ is uniformly integrable in $D_T$
  and hence relatively weakly compact in $L^1(D_T)$ by the
  Dunford-Pettis theorem, i.e. there exists $\xi \in L^1(D_T)$ such
  that $\beta_{\lambda'}(\cdot,X_{\lambda'})$ converges weakly to
  $\xi$ in $L^1(D_T)$.

  As a last step, we are going to show that $X_{\lambda'}$ converges
  to $X$ in the norm topology of $L^2(0,T;H)$, rather than just in its
  weak topology. Writing the regularized equation as in
  Lemma~\ref{lm:est2}, we have
  \[
    \frac{d}{dt}(X_{\lambda}-G\cdot Z)+A_\lambda X_{\lambda}+
    \beta_{\lambda}(\cdot, X_{\lambda})=0,
  \]
  where $A_\lambda X_\lambda=AJ_\lambda X_\lambda$ is bounded in
  $L^2(0,T;V')$ and $\beta_{\lambda}(\cdot, X_{\lambda})$ is bounded
  in $L^1(D_T)$. Therefore $\frac{d}{dt}(X_{\lambda} - G\cdot Z)$ is
  bounded in $L^1(0,T;V_0')$, and Simon's compactness criterion
  (see~\cite[Corollary~4, p.~85]{simon}) implies that
  $(X_{\lambda} - G \cdot Z)$ is relatively compact in $L^2(0,T;H)$.
  Since $G \cdot Z \in L^2(0,T;H)$ is independent of $\lambda$, the
  same conclusion holds for $(X_\lambda)$ and by uniqueness of the
  weak limit in $L^2(0,T;H)$ it immediately follows that $X_\lambda$
  converges to $X$ in $L^2(0,T;H)$.
\end{proof}

The last lemma provides us with a pair $(X,\xi)$ of (potentially
non-measurable) processes that serves as candidate solution to
\eqref{eq:add}.
\begin{proof}[Proof of Proposition~\ref{prop:aux}]
  We split the proof in several steps. We use the same symbols used in
  the proofs of the previous lemmata, without recalling their
  definitions explicitly.
  \smallskip\par\noindent
  \textsc{Step 1.} We are going to pass to the limit on each
  trajectory $\omega \in \Omega'$ in the regularized equation
  \[
    X_\lambda + \int_0^\cdot A_\lambda X_\lambda(s)\,ds
    + \int_0^\cdot \beta_\lambda(s,X_\lambda(s))\,ds
    = X_0 + G \cdot Z
  \]
  along the subsequence $\lambda'$. Let then $\omega$ be fixed and
  let us omit its explicit indication. By Lemma~\ref{lm:cvg} and the
  linearity of $A$, one has
  \[ 
    \int_0^t A_{\lambda'}X_{\lambda'}(s)\,ds \longto \int_0^t AX(s)\,ds
    \qquad \text{weakly in } V',
  \]
  hence also weakly in $V_0'$, for every $t \in [0,T]$.  Indeed, for
  any $\varphi \in V$ the map
  $\psi:=s \mapsto \mathbbm{1}_{[0,t]}(s) \varphi$ belongs to
  $L^2(0,T;V)$ and
  \begin{align*}
    \ip[\Big]{\varphi}{\int_0^t A_{\lambda'} X_{\lambda'}(s)\,ds}
    &= \int_0^T \ip[\big]{AJ_{\lambda'}X_{\lambda'}(s)}{\psi(s)}\,ds\\
    &\to \int_0^T\ip[\big]{AX(s)}{\psi(s)}\,ds
      = \ip[\Big]{\varphi}{\int_0^t AX(s)\,ds}.
  \end{align*}
  The same argument yields, choosing $\varphi \in L^\infty(D)$ or
  $\varphi \in V_0$, that
  \[
    \int_0^t \beta_{\lambda'}(s,X_{\lambda'}(s))\,ds \longto
    \int_0^t \xi(s)\,ds
  \]
  weakly in $L^1(D)$ and weakly in $V_0'$ for all $t \in
  [0,T]$. Therefore, for every $t \in [0,T]$, there exists
  $\tilde{X}(t) \in V_0'$ such that $X_{\lambda'}(t)$ converges to
  $\tilde{X}(t)$ weakly in $V_0'$. From this it easily follows that
  $X_{\lambda'}$ converges to $\tilde{X}$ weakly* in
  $L^\infty(0,T;V_0')$. In fact, for any $\psi \in L^1(0,T;V_0)$, one
  has
  $\ip[\big]{X_{\lambda'}(s)}{\psi(s)} \to
  \ip[\big]{\tilde{X}(s)}{\psi(s)}$ for a.a. $s \in [0,T]$, and
  \[
    \abs[\big]{\ip[\big]{X_{\lambda'}(s)}{\psi(s)}} \lesssim
    \norm[\big]{X_\lambda}_{L^\infty(0,T;H)} \norm[\big]{\psi(s)}_{V_0},
  \]
  where the right-hand side, as a function of $s$, belongs to
  $L^1(0,T)$. Then
  \[
    \int_0^T \ip[\big]{X_{\lambda'}(s)}{\psi(s)}\,ds \longto
    \int_0^T \ip[\big]{\tilde{X}(s)}{\psi(s)}\,ds
  \]
  by the dominated convergence theorem. However, since $X_{\lambda'}$
  converges to $X$ weakly* in $L^\infty(0,T;H)$, we infer that
  $X=\tilde{X}$ in $L^\infty(0,T;H)$. Therefore, taking the limit
  along $\lambda'$, we get
  \[
    X + \int_0^\cdot AX(s)\,ds + \int_0^\cdot \xi(s)\,ds =
    X_0 + G \cdot Z \qquad \text{ in } V_0'.
  \]
  This in turn implies that $X$ is c\`adl\`ag in $V_0'$, and since it
  also belongs to $L^\infty(0,T;H)$, it follows by
  Lemma~\ref{lm:strauss} that $X$ is weakly c\`adl\`ag in $H$.
  \smallskip\par\noindent
  \textsc{Step 2.} We are going to prove that
  $j(\cdot,X)+j^*(\cdot,\xi) \in L^1(D_T)$ and
  $\xi \in \beta(\cdot,X)$ a.e. in $D_T$. Since
  $\beta_{\lambda'}(X_{\lambda'})$ converges weakly to $\xi$ in
  $L^1(D_T)$, the weak lower semicontinuity of convex integrals (see,
  e.g., \cite[Theorem~2.3, p.~18]{Giaq:Mult}) immediately yields
  \[
    \int_{D_T} j^*(\xi)\,dx\,dt \leq \liminf_{\lambda' \to 0}
    \int_{D_T} j^*(\beta_{\lambda'}(X_{\lambda'}))\,dx\,dt
  \]
  where the right-hand side is finite by Lemma~\ref{lm:est2} (here and
  below we do not explicitly denote the dependence of $j$ and related
  maps on $\omega$ and $t$). Writing
  \[
    \lambda \beta_\lambda(X_\lambda) = X_\lambda -
    (I+\lambda\beta)^{-1}X_\lambda,
  \]
  the weak convergence of $\beta_\lambda(X_\lambda)$ is $L^1(D_T)$
  implies its boundedness, hence the left-hand side of the previous
  identity converges to zero in $L^1(D_T)$ along $\lambda'$. Moreover,
  as $X_{\lambda'}$ converges to $X$ in $L^2(0,T;H)$, it follows that
  $(I+\lambda'\beta)^{-1}X_{\lambda'}$ converges to $X$ in
  $L^1(D_T)$. Therefore, again by lower semicontinuity of convex
  integrals,
  \begin{equation}
    \label{eq:lsc}
    \int_{D_T} j(X)\,dx\,dt \leq \liminf_{\lambda' \to 0} 
    \int_{D_T} j\bigl( (I+\lambda'\beta)^{-1}(X_{\lambda'}) \bigr)\,dx\,dt,
  \end{equation}
  where the right-hand side is finite because the integrand is bounded
  by $\beta_{\lambda'}(X_{\lambda'})X_{\lambda'}$ (see part (iii) of
  the proof of Lemma~\ref{lm:est2}).

  Let $j_\lambda$ be the Moreau-Yosida regularization of $j$, i.e.
  \begin{align*}
    j_\lambda: \Omega \times [0,T] \times \erre 
    &\longto [0,+\infty\mathclose[\\
    (\omega,t,r) 
    &\longmapsto \inf_{s\in\erre} \Bigl( \frac{1}{2\lambda} \abs{r-s}^2
    + j(\omega,t,s) \Bigr).
  \end{align*}
  Recall that, for every $(\omega,t) \in \Omega \times [0,T]$,
  $j_\lambda(\omega,t,\cdot)$ is a convex differentiable function,
  with derivative equal to $\beta_\lambda(\omega,t,\cdot)$, that
  converges pointwise to $j(\omega,t,\cdot)$ from below. By definition
  of subdifferential one has, for any measurable set $E \subset D_T$,
  \[
    \int_E\beta_\lambda(\cdot,X_\lambda)(X_\lambda-z)\,dx\,dt\geq
    \int_E j_\lambda(\cdot,X_\lambda)\,dx\,dt- \int_E
    j_\lambda(\cdot,z)\,dx\,dt \qquad \forall z\in L^{\infty}(E).
  \]
  Since $X_{\lambda'} \to X_{\lambda'}$ in $L^2(0,T;H)$, there exists
  a subsequence of $\lambda'$, denoted by same symbol for simplicity,
  such that $X_{\lambda'} \to X$ a.e. in $D_T$. Therefore, thanks to
  the Severini-Egorov theorem, for every $\eta>0$ there exists
  $E_\eta \subseteq D_T$, with $\abs{D_T \setminus E_\eta} \leq \eta$,
  such that $X_{\lambda'} \to X$ uniformly on $E_\eta$. Choosing
  $E=E_\eta$ and passing to the limit along $\lambda'$ in the last
  inequality yields
  \[
    \int_{E_\eta} (X - z) \xi \,dx\,dt
    \geq \liminf_{\lambda' \to 0} \int_{E_\eta} 
    j_{\lambda'}(X_{\lambda'})\,dx\,dt - \int_{E_\eta} j(z) \,dx\,dt
    \qquad \forall z \in L^{\infty}(E_{\eta})
  \]
  because $\beta_{\lambda'}(X_{\lambda'})$ converges weakly to $\xi$
  in $L^1(D_T)$ and $X_{\lambda'}$ converges to $X$ uniformly on
  $D_T$, and $j_\lambda \leq j$. Moreover, by a well-known
  identity satisfied by the Moreau-Yosida regularization, one has
  \[
    j_\lambda(X_\lambda) = j\bigl( (I+\lambda\beta)^{-1}X_\lambda \bigr) 
    + \frac12 \lambda \bigl( X_\lambda - (I+\lambda\beta)^{-1}X_\lambda
    \bigr)^2.
  \]
  Since $(X_\lambda)$ is bounded in $L^2(D_T)$ and
  $(I+\lambda\beta)^{-1}$ is a contraction on $\erre$, it is easily
  seen that
  \[
    \lambda' \int_{D_T} \bigl( X_{\lambda'} -
    (I+\lambda'\beta)^{-1}X_{\lambda'} \bigr)^2\,dx\,dt \longto 0
  \]
  as $\lambda' \to 0$. By \eqref{eq:lsc} it then follows
  \[
    \int_{E_\eta} (X - z) \xi \,dx\,dt
    \geq \int_{E_\eta} \bigl( j(X) - j(z) \bigr) \,dx\,dt
    \qquad \forall z \in L^{\infty}(E_{\eta}).
  \]
  By a suitable choice of $z$, this implies
  \[
    (X-z)\xi \geq j(X) - j(z) \qquad \text{a.e. in } E_\eta
    \quad \forall z \in \erre
  \]
  (cf.~\cite{cm:div} for a detailed argument in a slightly
  simpler setting), and hence that $\xi \in \partial j(X) = \beta(X)$
  a.e. in $E_{\eta}$.  Since $\eta$ is arbitrary, it follows that
  $\xi \in \beta(X)$ a.e. in $D_T$.
  \smallskip\par\noindent
  \textsc{Step 3.} We are now going to show that the solution pair
  $(X,\xi)$ constructed in step 1 is unique. In particular, we claim
  that if there exist
  \[
    X_i\in L^{\infty}(0,T;H) \cap L^2(0,T;V), \quad \xi_i \in L^1(D_T),
    \qquad i=1,2,
  \]
  with $\xi_i \in \beta(\cdot, X_i)$ a.e. in $D_T$ and
  $j(\cdot,X_i) + j^*(\cdot,\xi_i) \in L^1(D_T)$ such that
  \[
    X_i + \int_0^\cdot AX_i(s)\,ds + \int_0^\cdot \xi_i(s)\,ds
    = X_0 + G \cdot Z,
  \]
  then $(X_1,\xi_1)=(X_2,\xi_2)$. In fact, setting $X := X_1 - X_2$ and
  $\xi := \xi_1 - \xi_2$, one has
  \[
    X + \int_0^\cdot AX(s)\,ds + \int_0^\cdot \xi(s)\,ds = 0,
  \]
  where $X\xi$ belongs to $L^1(D_T)$: in fact, $X\xi \geq 0$ by
  monotonicity of $\beta$ and, thanks to the convexity of $j$ and
  $j^*$ and to the hypothesis on their behavior at infinity, one has
  \begin{align*}
    \frac14 X\xi
    &\leq j(X/2) + j^*(\xi/2)
      = j\bigl( X_1/2 - X_2/2 \bigr) + j^*\bigl( \xi_1/2 - \xi_2/2 \bigr)\\
    &\lesssim 1+j(X_1)+j(X_2)+j^*(\xi_1)+j^*(\xi_2) \in L^1(D_T).
  \end{align*}
  By an argument completely analogous to the one used in the proof of
  Proposition~\ref{prop:Ito} (in fact easier), one obtains
  \[
    \norm{X}^2 + \int_0^\cdot\!\!\int_D X(s)\xi(s)\,dx\,ds \leq 0.
  \]
  Since the integrand in the previous identity is positive, it follows
  that $X=0$, which in turn implies that $\int_0^t \xi(s)\,ds = 0$ for
  all $t\in [0,T]$, hence also that $\xi = 0$, thus proving the claim.
  \smallskip\par\noindent
  \textsc{Step 4.} The uniqueness result proved in the previous step
  allows us to show that the collection of pairs $(X,\xi)$ indexed by
  $\omega \in \Omega'$ constructed in step 1 is in fact an optional
  process with values in $H \times L^1(D)$. This is far from obvious,
  mainly because $X$ and $\xi$ have been constructed, for each
  $\omega \in \Omega'$, as limits along subsequences $\lambda'$ that
  depend themselves on $\omega$. The crucial observation, which is an
  immediate consequence of the previous steps, is the following: from
  any subsequence of $\lambda$ one can extract a further subsequence
  $\lambda'$ (depending on $\omega$) such that the convergences of
  Lemma~\ref{lm:cvg} hold; but since the limits are unique, a
  classical result of elementary analysis ensures that the
  convergences hold along the original sequence $\lambda$, which is
  independent of $\omega$. As $X_\lambda$ converges to $X$ in
  $L^2(0,T;H)$ $\P$-almost surely and $(X_\lambda)$ is bounded in
  $L^2(\Omega;L^2(0,T;H))$, one has, passing to a subsequence if
  necessary, that $X_\lambda$ converges to $X$ weakly in
  $L^2(\Omega \times [0,T];H)$. Since $(X_\lambda)$ is also bounded in
  $L^2(\Omega \times [0,T];V)$, it follows that $X_\lambda$, again
  passing to a subsequence if necessary, converges weakly to $X$ in
  the latter space as well. Therefore there exists a sequence in the
  convex envelope of $(X_\lambda)$ that converges strongly to $X$ in
  $L^2(\Omega \times [0,T];V)$: since $X_\lambda$ is adapted and
  c\`adl\`ag with values in $H$, hence optional, for every
  $\lambda>0$, $X$ is an $H$-valued optional process. Completely
  analogously, $X$ is a (measurable) adapted $V$-valued process.
  In order to establish measurability properties of $\xi$, we need a
  more involved argument. Setting
  $\xi_\lambda := \beta_\lambda(\cdot, X_\lambda)$ for convenience,
  let $\phi \in L^\infty(D_T)$ and define
  \[
    \Xi_\lambda := \int_{D_T} \xi_\lambda \phi\,dx\,dt,
    \qquad \Xi := \int_{D_T} \xi \phi\,dx\,dt,
  \]
  so that $\Xi_\lambda$ converges to $\Xi$ $\P$-a.s.
  Jensen's inequality and part (iii) in the proof of
  Lemma~\ref{lm:est1} imply
  \[
    j^*(\cdot,\Xi_\lambda) \lesssim_{|D_T|,\phi}
    \int_{D_T} j^*(\cdot,\xi_\lambda) \,dx\,dt
    \leq \int_{D_T} \xi_\lambda X_\lambda \,dx\,dt,
  \]
  where the right-hand side, as a family indexed by $\lambda$, is
  bounded in $L^1(\Omega)$ by
  Lemma~\ref{lm:est1}. Lemma~\ref{valle_p_time} then implies that
  $(\Xi_\lambda)$ is uniformly integrable in $\Omega$ and hence, by
  Vitali's theorem, that $\Xi_\lambda$ converges to $\Xi$ in
  $L^1(\Omega)$. Again the estimate
  $j^*(\cdot, \xi_\lambda) \leq \xi_\lambda X_\lambda$ implies,
  recalling that the right-hand side, as a family indexed by
  $\lambda$, is bounded in $L^1(\Omega \times D_T)$, that
  $(\xi_\lambda)$ is uniformly integrable in $\Omega \times D_T$,
  hence relatively weakly compact as well, so that, by the
  Dunford-Pettis theorem, there exists
  $\tilde{\xi} \in L^1(\Omega \times D_T)$ such that $\xi_\lambda$
  converges weakly to $\tilde{\xi}$ in
  $L^1(\Omega \times [0,T];L^1(D))$, from which it follows, by a
  reasoning already used, that $\tilde{\xi}$ is an optional
  $L^1(D)$-valued process.
  For every $\lambda$ and $F \in \cF$ one has, setting
  $h:=\mathbbm{1}_F \in L^\infty(\Omega)$,
  \[
    \E h \Xi_\lambda =
    \int_{\Omega \times D_T} \xi_\lambda \phi h \,dx\,dt\,d\P,
  \]
  hence, passing to the limit as $\lambda \to 0$,
  \[
    \E h \Xi  = \int_{D_T} \bigl(\E h \xi\bigr) \phi \,dx\,dt
    = \int_{D_T} \bigl( \E h \tilde{\xi} \bigr) \phi \,dx\,dt.
  \]
  Therefore $\E\mathbbm{1}_F\xi = \E\mathbbm{1}_F\tilde{\xi}$ in
  $L^1(D_T)$ for every $F \in \cF$, i.e. $\xi=\tilde{\xi}$ in
  $L^1(D_T)$ $\P$-a.s.
  \smallskip\par\noindent
  \textsc{Step 5.} With the measurability properties of the processes
  $X$ and $\xi$ available, we can establish estimates of their
  moments. In fact, by the weak convergences of Lemma~\ref{lm:cvg} and
  the estimates of Lemma~\ref{lm:est1}, thanks to the weak and weak*
  lower semicontinuity of the norms, and to Fatou's lemma, it follows,
  writing $\xi_\lambda:=\beta_\lambda(\cdot,X_\lambda)$, that
  \begin{alignat*}{3}
    \E\norm[\big]{X}^2_{L^\infty(0,T;H)}
    &\leq \E\liminf_{\lambda \to 0} \norm[\big]{X_\lambda}^2_{L^\infty(0,T;H)}
      &&\leq \liminf_{\lambda \to 0} \E\norm[\big]{X_\lambda}^2_{L^\infty(0,T;H)},\\
    \E\norm[\big]{X}^2_{L^2(0,T;V)}
    &\leq \E\liminf_{\lambda \to 0} \norm[\big]{J_\lambda X_\lambda}^2_{L^2(0,T;V)}
      &&\leq \liminf_{\lambda \to 0}
      \E\norm[\big]{J_\lambda X_\lambda}^2_{L^2(0,T;V)},\\
    \E\norm[\big]{\xi}_{L^1(D_T)}
    &\leq \E\liminf_{\lambda \to 0} \norm[\big]{\xi_\lambda}_{L^1(D_T)}
    &&\leq \liminf_{\lambda\to 0} \E\norm[\big]{\xi_\lambda}_{L^1(D_T)}, 
  \end{alignat*}
  where the right-hand sides are all finite. Similarly, 
  the lower semicontinuity inequality
  \[
  \int_{D_T} \bigl( j(\cdot,X) + j^*(\cdot,\xi) \bigr)\,dx\,dt
    \leq \liminf_{\lambda \to 0}\int_{D_T} \bigl(
      j(\cdot,X_\lambda) + j^*(\cdot,\xi_\lambda) \bigr)\,dx\,dt
  \]
  yields, taking expectations on both sides and invoking Fatou's
  lemma,
  \begin{align*}
    \E\int_{D_T} \bigl( j(\cdot,X) + j^*(\cdot,\xi) \bigr)\,dx\,dt
    &\leq \E\liminf_{\lambda \to 0}\int_{D_T} \bigl(
      j(\cdot,X_\lambda) + j^*(\cdot,\xi_\lambda) \bigr)\,dx\,dt\\
    &\leq \liminf_{\lambda \to 0}
      \E\int_{D_T} \bigl( j(\cdot,X_\lambda) + j^*(\cdot,\xi_\lambda)
      \bigr)\,dx\,dt\\
    &\leq \liminf_{\lambda \to 0} \norm[\big]{%
      \xi_\lambda X_\lambda}_{L^1(\Omega \times D_T)},
  \end{align*}
  where the last term on the right-hand side is finite by
  Lemma~\ref{lm:est1}.
  \smallskip\par\noindent
  \textsc{Step 6.}  To conclude, let us show that the trajectories of
  $X$ are c\`adl\`ag in $H$. Proposition~\ref{prop:Ito} yields
  \begin{equation}
    \label{eq:cadlag}
    \begin{split}
      &\norm{X}^2 +2\int_0^\cdot \ip{AX(s)}{X(s)}\,ds
      +2\int_0^\cdot\!\!\int_D\xi(s)X(s)\,dx\,ds\\
      &\hspace{3em} = \norm{X_0}^2 + [G\cdot Z, G \cdot Z] + 2 (X_- G)
      \cdot Z,
    \end{split}
  \end{equation}
  where, by Fubini's theorem,
  \[
    \int_D \xi X\,dx \leq \int_D j(\cdot,X)\,dx
    + \int_D j^*(\cdot,\xi)\,dx \in L^1(0,T),
  \]
  thus also, taking into account that $X \in L^2(0,T; V)$
  and $AX \in L^2(0,T;V')$,
  \[
    \int_0^\cdot \ip{A X(s)}{X(s)}\,ds +
    \int_0^\cdot\!\!\int_D \xi(s)X(s)\,dx\,ds \in C([0,T]).
  \]
  Furthermore, the last term on the right-hand side of
  \eqref{eq:cadlag} is c\`adl\`ag, being a stochastic integral with
  respect to a semimartingale. Recalling the definition of quadratic
  variation, the same reasoning applies to the second term on the
  right-hand side of \eqref{eq:cadlag}. We deduce by inspection of
  \eqref{eq:cadlag} that the real-valued process $\norm{X}^2$ is
  c\`adl\`ag. Since $X$ is also weakly c\`adl\`ag in $H$ (see step 1)
  and $H$ is reflexive, we infer that the trajectories of $X$ are also
  strongly c\`adl\`ag in $H$. In fact, let $t \in [0,T\mathclose[$ and
  $(t_n)$ a sequence converging to $t$ from the right. Then
  $X(t_n) \to X(t)$ weakly in $H$ and $\norm{X(t_n)} \to \norm{X(t)}$
  imply that $X(t_n) \to X(t)$ in $H$. Similarly, if
  $t \in \mathopen]0,T]$ and $(t_n)$ is a sequence converging to $t$
  from the left, $X(t_n) \to X(t-)$ weakly in $H$ and
  $\norm{X(t_n)} \to \norm{X(t-)}$ yield $X(t_n) \to X(t-)$
  in $H$.
\end{proof}

In order to prove well-posedness of \eqref{eq:add} without the extra
regularity assumption \eqref{eq:G0} on the coefficient $G$, we prove
continuity, in a suitable sense, of the map $(X_0,G) \mapsto X$.
\begin{prop}
  \label{prop:cont}
  Let $(X_i,\xi_i)$, $i=1,2$, be strong solutions to \eqref{eq:add}
  with initial conditions $X_{0i} \in L^2(\Omega;H)$ and coefficients
  $G_i \in \cS_{C}(Z)$, respectively, where $C\in \cC(Z)$ and
  $\E\lambda_{T-}^{C}(G_i) < \infty$.  Then
  \begin{align*}
    &\E \bigl( X_1 - X_2 \bigr)_{T-}^{*2}
      + \E\int_0^T \norm[\big]{X_1(t)-X_2(t)}^2_{V}\,dt\\
  &\hspace{3em} \lesssim
    \E\norm[\big]{X_{01} - X_{02}}^2
    + \E \lambda^C_{T-}(G_1-G_2),
  \end{align*}
  where the implicit constant depends only on the coercivity constant
  of $A$.
\end{prop}
\begin{proof}
  Setting
  \begin{alignat*}{3}
    X &:= X_1 - X_2, &\quad  \xi &:= \xi_1-\xi_2,\\
    X_0 &:= X_{01} - X_{02}, &\quad  G &:= G_1-G_2,
  \end{alignat*}
  one has
  \[
    X + \int_0^\cdot AX(s)\,ds + \int_0^\cdot \xi(s)\,ds
    = X_0 + G \cdot Z.
  \]
  In analogy to a reasoning already used, the hypotheses on $j$ imply that
  \[
    \frac14 X \xi \leq 
    j(X/2) + j^*(\xi/2) \lesssim 1 + j(X_1) + j(X_2) + j^*(\xi_1) + j^*(\xi_2)
    \in L^1(D_T),
  \]
  so that Proposition~\ref{prop:Ito} yields
  \begin{align*}
    &\frac12 \norm{X}^2 + \int_0^\cdot \ip{AX(s)}{X(s)}\,ds
      + \int_0^\cdot\!\!\int_D \xi(s)X(s)\,dx\,ds\\
    &\hspace{3em} = \frac12 \norm{X_0}^2
      + \frac12 [G \cdot Z,G \cdot Z] + (X_- G) \cdot Z.
  \end{align*}
  Proceeding exactly as in the proof of Lemma~\ref{lm:est1}, one has
  \[
    \E [G\cdot Z,G \cdot Z]_{T-} \leq 5 \E \lambda_{T-}^C(G)
  \]
  and
  \[
    \E \bigl( (X_- G) \cdot Z \bigr)^*_{T-} \leq
    \frac14 \E X_{T-}^{*2} + 4 \E \lambda_{T-}^C(G),
  \]
  which immediately yield the claim by monotonicity and coercivity of
  $A$, and monotonicity of $\beta$.
\end{proof}

We are now in the position to prove Theorem~\ref{thm:add}.
\begin{proof}
  Let us set, for every $n \in \enne$, $G^n:=T_nG$. Then $G^n$ takes
  values in $\cL(K,V_0)$ and
  \[
    \norm[\big]{G^n}_{\cL(K,V_0)} \leq \norm[\big]{T_n}_{\cL(H,V_0)}
    \norm[\big]{G}_{\cL(K,H)},
  \]
  so that $G^n$ satisfies \eqref{eq:G0} for every $n \in \enne$.
  Moreover, by the uniform boundedness of $(T_n)$ in $\cL(H)$, one has
  \[
    \norm[\big]{G^n}_{\cL(K,H)} \leq \sup_{n \in \enne}
    \norm[\big]{T_n}_{\cL(H)} \norm[\big]{G}_{\cL(K,H)},
  \]
  so that, setting
  \[
    \bar{C} := \sup_{n \in \enne} \norm[\big]{T_n}_{\cL(H)} C
    \in \cC(Z),
  \]
  it follows that $G^n \in \cS_{\bar{C}}(Z)$ for every $n \in
  \enne$. Proposition~\ref{prop:aux} then ensures the existence and
  uniqueness of a strong solution $(X^n,\xi^n)$ to \eqref{eq:0} with
  data $(X_0,G^n)$ for every $n \in \enne$, i.e. such that
  \begin{equation}
    \label{eq:n}
    X^n + \int_0^\cdot AX^n(s)\,ds + \int_0^\cdot \xi_n(s)\,ds
    = X_0 + G^n \cdot Z.
  \end{equation}
  Furthermore, by inspection of the proof of Lemma~\ref{lm:est1} it
  follows that
  \[
  \E \norm[\big]{X^n}^2_{L^\infty(0,T;H)}
    + \E \norm[\big]{X^n}^2_{L^2(0,T; V)}
    + \E \norm[\big]{\xi^n X^n}_{L^1(D_T)}
  \lesssim \E\norm{X_0}^2 + \E\lambda_{T-}^{\bar{C}}(G^n),
  \]
  where the implicit constant is independent of $n$.
  In particular, since 
  \begin{align*}
  \lambda_{T-}^{\bar{C}}(G^n)
  &= \bar{C}(T-) \int_0^{T-} \norm[\big]{G^n(s)}^2_{\cL(K,H)}\,d\bar{C}(s)\\
  &\leq \sup_{n\in\enne} \norm[\big]{T_n}_{\cL(H)}^2
  \lambda_{T-}^{\bar{C}}(G) \in L^1(\Omega),
  \end{align*}
  there exists a constant $N$, independent of $n$, such that
  \[
    \E \norm[\big]{X^n}^2_{L^\infty(0,T;H)}
    + \E \norm[\big]{X^n}^2_{L^2(0,T; V)}
    + \E \norm[\big]{\xi^n X^n}_{L^1(D_T)} < N.
  \]
  Moreover, since $\bar{C}$ does not depend on $n$,
  Proposition~\ref{prop:cont} implies that
  \[
    \E \norm[\big]{X^{n_1} - X^{n_2}}^2_{L^\infty(0,T; H) \cap
      L^2(0,T;V)} \lesssim
    \E \lambda_{T-}^{\bar{C}} \bigl( G^{n_1}-G^{n_2} \bigr)
    \qquad \forall n_1,\,n_2 \in \enne.
  \]
  By the properties of $(T_n)_n$ and the dominated convergence
  theorem, the right-hand side converges to zero as $n_1,n_2 \to
  \infty$, hence the sequence $(X^n)$ is Cauchy in the space
  $L^2(\Omega;L^\infty(0,T;H)) \cap L^2(\Omega;L^2(0,T;V))$. As
  $X^n\xi^n = j(\cdot, X^n) + j^*(\cdot, \xi^n)$ and $j$ is positive,
  $\bigl(j^*(\cdot, \xi^n)\bigr)$ is bounded in
  $L^1(\Omega\times(0,T)\times D)$, hence, taking
  Lemma~\ref{valle_p_time} into account and arguing as in the proof of
  Lemma~\ref{lm:cvg}, it is easily seen that the sequence $(\xi^n)$ is
  relatively compact in $L^1(\Omega \times [0,T] \times D)$.
  Therefore, passing to a subsequence if necessary,
  \begin{alignat*}{3}
    X^n &\longto X &&\quad \text{ in } L^2(\Omega;L^\infty(0,T;H))
    \cap L^2(\Omega;L^2(0,T;V)),\\
    \xi^n &\longto \xi &&\quad \text{ weakly in } L^1(\Omega \times D_T).
  \end{alignat*}
  The first convergence implies that 
  \[
  \int_0^\cdot AX^n(s)\,ds \longto \int_0^\cdot AX(s)\,ds \qquad
  \text{ in } L^2(\Omega;C([0,T];V')),
  \]
  and that $\bigl( X^n-X \bigr)^*_{T-} \to 0$ in
  $L^2(\Omega)$, because $X^n$ has c\`adl\`ag trajectories for each
  $n \in \enne$ thanks to Proposition~\ref{prop:aux}. In particular, $X$
  has c\`adl\`ag trajectories as well.
  The uniform boundedness of $(T_n)$ in $\cL(H)$ and the dominated
  convergence theorem for stochastic integrals yield
  \[
    \bigl( G^n \cdot Z - G \cdot Z \bigr)_{T-}^* \longto 0
    \qquad \text{ in } L^2(\Omega).
  \]
  From $\Delta X^n(T) = \Delta(G^n \cdot Z)_T = G^n_T \Delta Z_T$ and
  the above uniform convergences up to $T-$ it immediately follows
  that
  \[
  \bigl( X^n-X \bigr)^*_{T} \longto 0, \qquad
  \bigl( G^n \cdot Z - G \cdot Z \bigr)^*_{T} \longto 0
  \]
  in $L^0(\Omega)$ as $n \to \infty$.
  Let $\phi \in V_0$ and $F \in \cF$. Recalling that
  $\xi^n \to \xi$ weakly in $L^1(\Omega \times D_T)$, one has
  \[
    \E\mathbbm{1}_F \ip[\Big]{\phi}{\int_0^t \xi^n(s)\,ds}
    \longto \E\mathbbm{1}_F \ip[\Big]{\phi}{\int_0^t \xi(s)\,ds}
    \qquad \forall t \in [0,T].
  \]
  Taking the duality product of both sides of \eqref{eq:n} with
  $\phi \in V_0$ and multiplying by $\mathbbm{1}_F$, one readily
  infers, passing to the limit as $n \to \infty$ and taking into
  account that $\varphi$ and $F$ are arbitrary, that
  \[
    X(t) + \int_0^t AX(s)\,ds
    + \int_0^t \xi(s)\,ds = X_0 + \bigl( G \cdot Z \bigr)_t
    \qquad \forall t \in [0,T]
  \]
  as an identity in $V_0'$. Since both sides of the
  equality are immediately seen to be c\`adl\`ag (with values in
  $V_0'$), it follows that equality holds in $V_0'$ also in the sense
  of indistinguishability, not only in the sense of modifications. By
  comparison, the identity also holds in $V' \cap L^1(D)$.  Moreover,
  arguing as in step 2 of the proof of Proposition~\ref{prop:aux}, we
  deduce that $\xi\in\beta(X)$ a.e.~in $\Omega\times(0,T)\times D$.
  The uniqueness of $(X,\xi)$ follows by an argument completely
  analogous to the one used in step 3 of the proof of
  Proposition~\ref{prop:aux}, appealing to the integration-by-parts
  formula of Proposition~\ref{prop:Ito}.
\end{proof}

Suitably localized versions of the previous results hold.
\begin{prop}
  Let $\tau \neq 0$ be a stopping time with $\tau \leq T$, $C$ a
  control process for $Z$, and $G$ a strongly predictable process such
  that $\E\lambda_{\tau-}^C(G) < \infty$.  If
  $X_0 \in L^2(\Omega; H)$, then \eqref{eq:add} admits a unique strong
  solution on $\cc{0}{\tau}$.
\end{prop}
\begin{proof}
  Let us consider the equation
  \begin{equation}
    \label{eq:tau}
    d\tilde{X} + A\tilde{X}\,dt + \beta(\tilde{X})\,dt
    \ni G\, dZ^{\tau-}, \qquad X(0)=X_0,
  \end{equation}
  where $Z^{\tau-}$ is a semimartingale with control process $C^{\tau-}$
  (see Lemma~\ref{lm:stopped}). Since
  \begin{align*}
    \lambda_{T-}^{C^{\tau-}}(G)
    = C^{\tau-}_{T-} \bigl( \norm{G}^2 \cdot C^{\tau-} \bigr)_{T-}
      = C_{\tau-} \bigl( \norm{G}^2 \cdot C \bigr)_{\tau-}
      = \lambda^C_{\tau-}(G)
  \end{align*}
  where the expectation of the last term is finite by assumption,
  equation \eqref{eq:tau} admits a unique strong solution
  $(\tilde{X},\tilde{\xi})$. In particular,
  \[
    \tilde{X} + \int_0^\cdot A\tilde{X}(s)\,ds
    + \int_0^\cdot \tilde{\xi}(s)\,ds = X_0 + G \cdot Z^{\tau-},
  \]
  which implies that $\Delta\tilde{X}_{\tau}=0$, because the Lebesgue
  integrals and the stochastic integral have no jump at $\tau$.
  Setting $X = \tilde{X}$ on $\co{0}{\tau}$ and
  $X_\tau:=X_{\tau-}+G_\tau\Delta Z_\tau$, and
  $\xi := \mathbbm{1}_{\cc{0}{\tau}} \tilde{\xi}$, we are left with
  \[
    X^\tau + \int_0^{\cdot \wedge \tau} AX(s)\,ds
    + \int_0^{\cdot \wedge \tau} \xi(s)\,ds = X_0
    + (G \mathbbm{1}_{\cc{0}{\tau}}) \cdot Z,
  \]
  i.e. $(X,\xi)$ is a strong solution on $\cc{0}{\tau}$ to
  \eqref{eq:add}. Since $\tilde\xi\in\beta(\tilde X)$
  a.e~in $\Omega\times (0,T)\times D$, we have in particular 
  that $\xi\in\beta(X)$ a.e.~in $[\![0,\tau[\![\times D$.
  To prove uniqueness it suffices to note that a strong solution
  $(X,\xi)$ on $\cc{0}{\tau}$ to \eqref{eq:add} coincides on
  $\co{0}{\tau}$ with the restriction to $\co{0}{\tau}$ of the unique
  strong solution $(\tilde{X},\tilde{\xi})$ to
  \eqref{eq:tau}. Uniqueness on the closed stochastic interval
  $\cc{0}{\tau}$ follows by the definition of $X_\tau$.
\end{proof}

As an immediate consequence of the uniqueness argument just used, one
obtains that (strong) solutions on closed stochastic intervals form a
direct system, in the following sense: if $(X,\xi)$ is a solution on
$\cc{0}{\tau}$ to \eqref{eq:add} and $\sigma$ is a stopping time with
$\sigma \leq \tau$, it is easily seen that $(X^\sigma,\xi^\sigma)$ is
a solution on $\cc{0}{\sigma}$ to \eqref{eq:add}. Such a solution, by
the reasoning of the previous remark, is the unique solution on
$\cc{0}{\sigma}$. This also implies that, given $(X_1,\xi_1)$ solution
on $\cc{0}{\tau_1}$ and $(X_2,\xi_2)$ solution on $\cc{0}{\tau_2}$,
one can construct a solution $(X,\xi)$ on $\cc{0}{\tau_1 \vee \tau_2}$
setting
\[
(X,\xi) :=
\begin{cases}
  (X_1,\xi_1) & \text{ on } \cc{0}{\tau_1},\\
  (X_2,\xi_2) & \text{ on } \cc{0}{\tau_2}.
\end{cases}
\]

\begin{prop}
  \label{prop:loc}
  Let $(X_i,\xi_i)$ be strong solutions on $\cc{0}{\tau_i}$, $i=1,2$,
  to \eqref{eq:add} with initial conditions $X_{0i} \in L^2(\Omega;H)$
  and coefficients $G_i \in \cS_{C}(Z)$, respectively, where $C$ is a
  control process for the semimartingale $Z$ and
  $\E\lambda_{\tau_i-}^{C}(G_i) < \infty$. Setting
  $\tau := \tau_1 \wedge \tau_2$, one has
  \begin{align*}
    &\E \bigl( X_1 - X_2 \bigr)_{\tau-}^{*2}
      + \E\int_0^\tau \norm[\big]{X_1(t)-X_2(t)}^2_{V}\,dt\\
  &\hspace{3em} \lesssim
    \E\norm[\big]{X_{01} - X_{02}}^2
    + \E \lambda^C_{\tau-}(G_1-G_2).
  \end{align*}
\end{prop}
\begin{proof}
  By the above discussion about strong solutions on closed stochastic
  intervals forming a direct system, it is immediately seen that
  $(X_1,\xi_1)$ and $(X_2,\xi_2)$ are strong solutions on
  $\cc{0}{\tau}$, as well as that
  $(X_i^{\tau-},\xi_i^{\tau-}) =
  (\tilde{X}_i^{\tau-},\tilde{\xi}_i^{\tau-})$, where
  $(\tilde{X}_i,\tilde{\xi}_i)$ is the unique strong solution to
  \[
    d\tilde{X}_i + A\tilde{X}_i\,dt + \beta(\cdot,\tilde{X}_i)\,dt \ni
    G_i\,dZ^{\tau-}, \qquad \tilde{X}_i(0)=X_{0i}.
  \]
  Since
  $\E\lambda^{C^{\tau-}}_{T-}\bigl(G_1-G_2 \bigr) =
  \E\lambda^C_{\tau-}(G_1 - G_2)$ and $C^{\tau-}$ is a control process
  for $Z^{\tau-}$ by Lemma~\ref{lm:stopped},
  Proposition~\ref{prop:cont} yields
  \begin{align*}
    &\E \bigl( X_1 - X_2 \bigr)_{\tau-}^{*2}
    + \E\int_0^\tau \norm[\big]{X_1(t)-X_2(t)}^2_{V}\,dt\\
    &\hspace{3em} \lesssim \E\norm[\big]{X_{01} - X_{02}}^2 
    + \E\lambda^C_{\tau-}(G_1 - G_2). \qedhere
  \end{align*}
\end{proof}


\ifbozza\newpage\else\fi
\section{Well-posedness with multiplicative noise}
\label{sec:mult}
This section is devoted to the proof of Theorem~\ref{thm:1}. We begin
showing that strong solutions on closed stochastic intervals exist.
\begin{prop}
  \label{prop:le}
  There exists a stopping time $\tau \neq 0$ and a strong
  solution on $\cc{0}{\tau}$ to \eqref{eq:0}.
\end{prop}
\begin{proof}
  Let $\alpha \in \mathopen]0,1\mathclose[$ be a constant to be chosen
  later, $C$ a control process for $Z$, and $\tau^0$ the stopping time
  defined as
  \[
    \tau^0 := \inf\bigl\{ t \in [0,T]: C_t (L_t - L_0) \geq \alpha
    \bigr\} \wedge T.
  \]
  Note that $\tau^0$ is well-defined and not identically $0$ as the
  process $C(L - L_0)$ starts from $0$ and is right-continuous. Let
  $R \in \erre_+$ be such that the event $\{\norm{X_0} \leq R\}$ has
  strictly positive probability, and set $\tau:=\tau^0
  \mathbbm{1}_F$. Since $F \in \cF_0$, it is easily seen that $\tau$
  is a stopping time.
  Let $\mathbb{S}^2(T-)$ denote the vector space of adapted c\`adl\`ag
  processes $Y:\Omega \times [0,T\mathclose[ \to H$ such that
  \[
    \norm{Y}_2:= \bigl( \E Y_{T-}^{*2} \bigr)^{1/2} < \infty.
  \]
  It is not difficult to see that $\mathbb{S}^2(T-)$, endowed with the
  norm $\norm{\cdot}_2$, is a Banach space.
  For every $Y \in \mathbb{S}^2(T-)$ one has
  \begin{align*}
    &\lambda_{T-}^{C^{\tau-}} \bigl( B(Y) \bigr) = C^{\tau-}_{T-}
      \int_0^{T-} \norm[\big]{[B(Y)](s)}_{\cL(K,H)}^2\,dC^{\tau-}_s\\
    &\hspace{3em} \leq C_{\tau-} (L_{\tau-} - L_0) %
      \bigl( 1 + Y_{T-}^{*2} \bigr) \leq 
      \alpha \bigl( 1 + Y_{T-}^{*2} \bigr)
      \in L^1(\Omega),
  \end{align*}
  so that the equation
  \[
    d\tilde{X}(t) + A\tilde{X}(t)\,dt + \beta(\tilde{X}(t))\,dt
    \ni B(Y)\,dZ^{\tau-}_t, \qquad X(0)=X_0,
  \]
  admits a unique strong solution $(\tilde{X},\tilde{\xi})$ by
  Theorem~\ref{thm:add} (by the definition of the stopping time
  $\tau$, the latter result is indeed applicable).
  In particular, the map $Y \mapsto
  \tilde{X}$ is a homomorphism of $\mathbb{S}^2(T-)$. Moreover, for
  any $Y_1$, $Y_2 \in \mathbb{S}^2(T-)$, Proposition~\ref{prop:cont}
  yields, with obvious meaning of the notation,
  \[
    \norm[\big]{\tilde{X}_1 - \tilde{X}_2}^2_2
    + \E \norm[\big]{\tilde{X}_1 - \tilde{X}_2}^2_{L^2(0,T;V)} \lesssim
    \E \lambda_{T-}^{C^{\tau-}} \bigl(
    B(Y_{1}) - B(Y_{2}) \bigr),
  \]
  where, by the Lipschitz assumption on $B$,
  \begin{align*}
    &\lambda_{T-}^{C^{\tau-}} \bigl( B(Y_{1}) - B(Y_{2}) \bigr)\\
    &\hspace{3em} = C^{\tau-}_{T-} \int_0^{T-}
      \norm[\big]{[B(Y_{1})](s) - [B(Y_{2})](s)}^2_{\cL(K,H)} \,dC^{\tau-}_s\\
    &\hspace{3em} \leq
      C(\tau-)\int_0^{T-} \bigl( Y_1 - Y_2 \bigr)^{*2}_{s-}
      \,dL^{\tau-}(s) \leq
      \alpha \bigl( Y_1 - Y_2 \bigr)^{*2}_{T-},
  \end{align*}
  which implies
  \[
    \norm[\big]{\tilde{X}_1 - \tilde{X}_2}^2_2
    + \E \norm[\big]{\tilde{X}_1 - \tilde{X}_2}^2_{L^2(0,T;V)}
    \lesssim \alpha \norm[\big]{Y_1-Y_2}^2_2.
  \]
  Choosing $\alpha$ small enough, $Y \mapsto \tilde{X}$ is a
  contraction of $\mathbb{S}^2(T-)$, hence it admits a unique fixed
  point $\tilde{X} \in \mathbb{S}^2(T-)$ (the abuse of notation is
  harmless). Setting $X := X_0 \mathbbm{1}_{\{\tau=0\}} + \tilde{X}$
  in $\co{0}{\tau}$, $X_\tau := X_{\tau-} +
  [B(\tilde{X})]_{\tau}\Delta Z_\tau$, and $\xi := \tilde{\xi}
  \mathbbm{1}_{\cc{0}{\tau}}$, it is immediately seen that $(X,\xi)$
  is a strong solution on $\cc{0}{\tau}$ to \eqref{eq:0}.
\end{proof}

Once existence of solutions on stochastic intervals is established, we
establish their uniqueness in a local sense.
\begin{lemma}
  \label{lm:lu}
  Let $(X_1,\xi_1)$ and $(X_2,\xi_2)$ be strong solutions to
  \eqref{eq:0} on $\cc0{\tau_1}$ and $\cc0{\tau_2}$, respectively.
  Then, setting $\tau:=\tau_1\wedge\tau_2$, one has $X_1=X_2$ and
  $\xi_1=\xi_2$ on $\cc{0}{\tau}$.
\end{lemma}
\begin{proof}
  Setting $X:=X_1-X_2$ and $\xi:=\xi_1-\xi_2$, one has
  \begin{equation}
    \label{eq:ds}
    X^\tau + \int_0^\cdot \mathbbm{1}_{\cc0\tau}AX(s)\,ds 
    + \int_0^\cdot \mathbbm{1}_{\cc0\tau}\xi(s)\,ds =
    \bigl( \mathbbm{1}_{\cc0\tau} (B(X_{1}) - B(X_{2}) \bigr) \cdot Z,
  \end{equation}
  where $B(X_{1}) \in \cS_{C_1}(Z)$, $B(X_{2}) \in \cS_{C_2}(Z)$, with
  $C_1$ and $C_2$ control processes for $Z$. Recalling that
  $C:=C_1+C_2$ is a control process for $Z$, let us set, for every
  $k \in \enne$,
  \[
    \tau^0_k := \inf \bigl\{ t \in [0,T]:\, C(t) (L(t)-L(0)) \geq k \bigr\}
    \wedge \tau
  \]
  and $\tau_k := \tau^0_k \mathbbm{1}_{F_k}$, where $F_k$ is the event
  $\{\norm{X_0} \leq k \}$. By the hypotheses on $B$ it follows that
  \begin{align*}
    \lambda_{\tau_k-}^C \bigl( B(X_{i}) \bigr)
    &= C_{\tau_k-} \int_0^{\tau_k-}
      \norm[\big]{[B(X_{i})](s)}^2_{\cL(K,H)}\,dC_s\\
    &\leq C_{\tau_k-} \int_0^{\tau_k-} \bigl( 1 + (X_i)^{*2}_{s-}
      \bigr)\,dL_s\\
    &\leq C_{\tau_k-} (L_{\tau_k-} - L_0) \bigl( 1 + (X_i)^{*2}_{\tau_k-} \bigr)\\
    &\leq k \bigl( 1 + (X_i)^{*2}_{\tau_k-} \bigr) \in L^1(\Omega).
  \end{align*}
  Hence, for every stopping time $\sigma \leq \tau_k$,
  Proposition~\ref{prop:cont} yields
  \[
    \E X_{\sigma-}^{*2} + \E\int_0^\sigma \norm{X(s)}^2_V\,ds
    \lesssim \E \lambda^C_{\sigma-} \bigl( B(X_{1}) - B(X_{2}) \bigr),
  \]
  thus also, by the Lipschitz continuity of $B$,
  \[
    \E \bigl( X_1-X_2 \bigr)_{\sigma-}^{*2} \lesssim
    k \E \bigl( (X_1-X_2)^{*2} \cdot L \bigr)_{\sigma-},
  \]
  which implies, by Lemma~\ref{lm:gron}, that
  $\E\bigl( X_1-X_2 \bigr)^{*2}_{\tau_k-} = 0$ for every $k \in \enne$.
  Since $\tau_k$ tends monotonically to $\tau$ as $k \to \infty$, it
  immediately follows that $X_1=X_2$ on $\co{0}{\tau}$. This implies
  that $B(X_1)=B(X_2)$ on $\cc{0}{\tau}$, hence the jumps at $\tau$ of
  $X_1$ and $X_2$ are both equal to $[B(X_1)]_\tau \Delta Z_\tau$, so
  that $X_1=X_2$ on $\cc{0}{\tau}$. Finally, by comparison in
  \eqref{eq:ds}, one gets $\int_0^\cdot\xi(s)\,ds=0$, which implies
  also $\xi_1=\xi_2$.
\end{proof}

\medskip

Let us now come to the core of the proof of Theorem~\ref{thm:1}. The
idea is simply to iterate the construction of
Proposition~\ref{prop:le}, to obtain a solution on a sequence of
stochastic intervals $\cc{\tau_n}{\tau_{n+1}}$, $n \in \enne$, and to
show that $\P(\tau_n<T)$ tends to zero as $n \to \infty$. Calling
$\tau_1$ the stopping time given by Proposition~\ref{prop:le}, let us
define the increasing sequence of stopping times
$(\tau_n)_{n \in \enne}$ defined as
\[
\tau_{n+1} :=
\begin{cases}
  \tau_n, &\text{if } \norm{X(\tau_n)} > n,\\[4pt]
  \inf \bigl\{ t \geq \tau_n:\; C_t(L_t-L_{\tau_n}) > \alpha \bigr\}
  \wedge T, &\text{if } \norm{X(\tau_n)} \leq n,
\end{cases}
\]
where $\alpha$ is a constant as chosen in the proof of
Proposition~\ref{prop:le}. Note that $\tau_{n+1}$ is indeed a stopping
time because the event $\{\norm{X(\tau_n)} > n\}$ belongs to
$\cF_{\tau_n}$.
Proposition~\ref{prop:le} yields the existence of a strong solution on
$\cc{\tau_n}{\tau_{n+1}}$ to equation \eqref{eq:0} started at
$\tau_n$. A standard patching argument shows that one thus obtains a
strong solution $(X_n,\xi_n)$ on $\cc{0}{\tau_n}$ for every
$n \in \enne$.

We are going to show that $\P(\lim_n \tau_n < T) = 0$. Assume, by
contradiction, that $\P(\lim_n \tau_n < T) > 0$. One can rule out that
$\tau_{n+1} \neq \tau_n$ occurs only a finite number of times. In
fact, if it were the case, then there would exist $\bar{n} \in \enne$
such that $\norm{X(\tau_{\bar{n}})}$ is larger than every natural
number on an event of positive probability. This is impossible,
because $X_{\tau_n}$ is a well-defined $H$-valued random variable for
all $n \in \enne$. This implies that, on an event $F$ of strictly
positive probability, $L_{\tau_{n+1}} - L_{\tau_n} > 0$ for every $n$
belonging to an infinite subset $\enne'$ of $\enne$. Since $C$ is
increasing, one has
\[
  L_{\tau_{n+1}} - L_{\tau_n} > \frac{\alpha}{C_{\tau_n}}
  \geq \frac{\alpha}{C_T} \qquad \forall n \in \enne',
\]
hence denoting the variation of $L$ by $\abs{L}$ and recalling that
$L$ is also increasing,
\[
  \abs{L} \geq \sum_{n \in \enne'} \abs[\big]{L_{\tau_{n+1}} -
    L_{\tau_n}} = \infty \qquad \text{ on } F.
\]
This contradicts the hypotheses on $L$, therefore $\tau_n \to T$
$\P$-a.s. as $n \to \infty$.
The solution constructed above is thus defined on the whole interval
$[0,T]$. Furthermore, such a solution is also unique, thanks to
Lemma~\ref{lm:lu}. 

An argument entirely analogous to the one used in the proof of
Lemma~\ref{lm:lu} yields, bearing in mind the definition of $\tau_n$,
\[
  \E X^{*2}_{\tau_n-} +
  \E\int_0^{\tau_n} \norm{X(s)}^2_V\,ds
  +\E\int_0^{\tau_n}\!\!\int_D \xi(s)X(s)\,dx\,ds \lesssim  n^2
  \qquad \forall n \in \enne,
\]
hence, in particular,
\[
  X^{*2}_{\tau_n-}
  + \int_0^{\tau_n} \norm{X(s)}^2_V\,ds
  + \int_0^{\tau_n}\!\!\int_D \xi(s)X(s)\,dx\,ds
\]
is finite $\P$-a.s. for all $n \in \enne$. Since
$X^*_{\tau_n} \leq X^*_{\tau_n-} + \norm{\Delta X(\tau_n)}$ and, for
all $\omega$ in an event of probability one, there exists $\bar{n}$
such that $\tau_n(\omega)=T$ for all $n \geq \bar{n}$, it follows that
\[
  X^{*2}_{T}
  + \int_0^{T} \norm{X(s)}^2_V\,ds
  + \int_0^{T}\!\!\int_D \xi(s)X(s)\,dx\,ds < \infty
\]
with probability one.

Let us now turn to the continuity with respect to the initial
datum. Let $(X_{0n})$ be a sequence of $\cF_0$-measurable random
variables such that $X_{0n} \to X_0$ in probability, and let $X_n$ be
the unique solution to \eqref{eq:0} with initial datum $X_{0n}$.
Then there exists a subsequence $(X_{0n'})$ converging to $X_0$
$\P$-almost surely.
Setting
\[
S_k := \bigcap_{n' \geq k} \bigl\{ \norm{X_{0n'}-X_0} \leq 1 \bigr\},
\]
it is clear that $(S_k)$ is an increasing sequence of elements of
$\cF_0$ whose limit as $k \to \infty$ is an event of probability one. In fact,
\[
\P(S_k) = \P\bigl( \norm{X_{0n'}-X_0} \leq 1 \; \forall n' \geq k \bigr),
\]
which converges to one as $k \to \infty$ by definition of 
almost sure convergence. 
Moreover, $(X_{0n}-X_0) \mathbbm{1}_{S_k}$ obviously
converges to zero in probability as $n \to \infty$ for every $k$, and
\[
\norm[\big]{(X_{0n}-X_0) \mathbbm{1}_{S_k}} \leq 1 \qquad \forall n \geq k.
\]
Therefore, by the dominated convergence theorem, $(X_{0n}-X_0)
\mathbbm{1}_{S_k}$ converges to zero in $L^2(\Omega;H)$ as $n \to
\infty$ for each $k$.
Let $(\tau_k)$ be an increasing sequence of stopping times converging
to $T$, for instance as the one constructed above, and define a new
sequence of stopping times $(\sigma_k)$ as $\sigma_k:=\tau_k
\mathbbm{1}_{S_k}$. Then a (by now) familiar reasoning using It\^o's
formula for the square of the norm, stopping at $\sigma_k-$, and
applying the stochastic Gronwall lemma, much as in the proof of
Lemma~\ref{lm:lu}, yields
\[
  \E\bigl( X - X_n \bigr)^{*2}_{\sigma_k-} 
  +\E\int_0^{\sigma_k}\norm{(X - X_n)(s)}_V^2\,ds \lesssim
  \E \norm[\big]{X_{0}-X_{0n}}^2 \mathbbm{1}_{S_k},
\]
where the right-hand side converges to zero as $n \to \infty$ for
every $k$. We have thus shown that $X_n$ converges to $X$ prelocally
in $\mathbb{S}^2(T)$. Since $T$ was arbitrary and all results continue
to hold if $T$ is replaced by, e.g., $T+1$, $X^n$ converges to $X$
prelocally also in $\mathbb{S}^2((T+1))$, which implies that
$(X_n-X)^*_T$ converges to zero in probability (see, e.g.,
\cite[p.~261]{Protter}).
The proof of Theorem~\ref{thm:1} is thus completed.


\ifbozza\newpage\else\fi
\bibliographystyle{amsplain}
\bibliography{ref,extra}

\end{document}